\PassOptionsToPackage{dvipsnames}{xcolor}
\documentclass[11pt,review]{amsart}

\usepackage[utf8]{inputenc}

\usepackage[margin=1in]{geometry}

\usepackage{amsmath, amsthm, amssymb}
\usepackage{mathtools}
\usepackage{mathrsfs}
\usepackage{stmaryrd}
\usepackage{stix}

\usepackage{siunitx}

\usepackage{graphicx}
\graphicspath{{fig/}}

\usepackage{booktabs}
\usepackage{multirow}

\usepackage{float}
\usepackage{caption}
\usepackage{subcaption}

\usepackage{algorithm}
\usepackage{algorithmic}

\usepackage{tikz}
\usepackage{standalone}

\usepackage{calligra}
\usepackage{oubraces}

\usepackage{soul}   

\usepackage{comment}

\usepackage[sort&compress,numbers,square]{natbib}
\usepackage{hyperref}
\hypersetup{
  colorlinks=true,
  linkcolor=Purple,
  citecolor=Orange,
  urlcolor=Purple
}

\theoremstyle{plain}
\newtheorem{theorem}{Theorem}[section]

\newtheorem{lemma}[theorem]{Lemma}

\theoremstyle{definition}

\newtheorem{Remark}[theorem]{Remark}

\newcommand\calP{{\mathcal{P}}}
\newcommand\calT{{\mathcal{T}}}

\newcommand{\mesh}{\calP}

\newcommand{\mean}[1]{\{\kern-1.1mm\{#1\}\kern-1.1mm\}}          

\newcommand{\ndg}[1]{| \kern -.25mm \|{#1}| \kern -.25mm \|}
\newcommand{\ndgp}[1]{| \kern -.25mm \|{#1}| \kern -.25mm \|_{\rm \mesh}}
\newcommand{\ndgps}[1]{| \kern -.25mm \|{#1}| \kern -.25mm \|_{\rm \mesh,s}}
\newcommand{\C}{\mathbb{C}}
\newcommand{\R}{\mathbb{R}}
\newcommand{ \Clip}{C_{\textrm{lip}}}
\newcommand{ \Cemb}{C_{\textrm{emb}}}

\newcommand{\ncdg}[1]{| \kern -.25mm \|{#1}| \kern -.25mm \|_{\rm DG}}

\newcommand{\real}{\mathrm{Re}}

\newcommand{\bnabla}{\nabla_{\kern-.05cm\mesh}^{}}

\title[Induction Heating Problem]{Analysis and approximation of a \\ two-dimensional induction heating problem}

\author[G.R.~Barrenechea]{Gabriel R. Barrenechea}
\address[G.R.~Barrenechea]{Department of Mathematics and
    Statistics, University of Strathclyde, 26 Richmond Street,
    Glasgow, G1 1XH United Kingdom}
\email{gabriel.barrenechea@strath.ac.uk}

\author[K.~MacKenzie]{Katherine MacKenzie}
\address[K.~MacKenzie]{Institute for Mathematical Innovation, University of Bath, Building 4 East South, Room 4.10, Bath, BA2 7AY}
\email{km2560@bath.ac.uk}

\author[A.J.~Salgado]{Abner J. Salgado}
\address[A.J.~Salgado]{Department of Mathematics, University of Tennessee, Knoxville, TN 37996, USA}
\email{asalgad1@utk.edu}

\begin{document}
\maketitle

\begin{abstract}
In this paper, we analyse the existence of solutions and finite element approximation of a steady-state two-dimensional induction heating problem.  One of the main difficulties of the problem is its right-hand side which, at a first sight, is only integrable.  Using a priori regularity results for the PDEs involved it is shown that the natural weak formulation  of the problem can be justified. Then, we study the finite element approximation and prove that the standard Galerkin FEM converges in convex domains and under suitable conditions on the mesh. We improve on this result by applying the recently-proposed bound-preserving method (BPM) to the heat equation, and show that this method converges to a solution of the  problem under less stringent conditions on the domain and the mesh.  As these analyses are carried out without any assumption on regularity of the solutions, then the convergence of the finite element method also proves existence of solutions. Several numerical experiments confirm the theoretical results, and showcase the improvement provided by the use of the bound-preserving method over the standard finite element method.
\end{abstract}

\section{Introduction}
\label{sec:intro}

\subsection*{Description of the Problem}

 Induction heating is a method commonly used in industry for heating electrically conductive materials using a magnetic field.  The physical process can be described as follows: a solid piece of metal (the billet)
passes through an array of  coils that generate a magnetic field that heats the billet to the desired temperature.  This is a process
widely used in engineering and materials science, and has received a wide attention both in the engineering and applied mathematics
community (see, e.g., the very comprehensive monograph \cite{Touzani-Rappaz-2014}).

The full model consists of the time-dependent Maxwell's equation in the air and billet, coupled with the heat equation inside the billet.
In its full generality, the problem is time-dependent, posed on a moving domain, the coupling is nonlinear, includes nonlinear radiation boundary conditions on the boundary of
the billet, and the material coefficients can depend on the temperature and magnetic field. Thus, the numerical
approximation (and the mathematical analysis) of such a problem is a challenging task.  As a consequence
several simplifications of the model have been proposed.  Probably the most basic one consists of reducing the
problem to a two-dimensional setting,  where the billet is considered infinite and the magnetic
field parallel to the billet itself.  In this case, the problem is reduced to a nonlinear PDE system inside the
billet. More precisely,  let $\Omega \in \R^2$ be an open, bounded  polygon with boundary $\partial \Omega$. Then the steady-state simplified solenoidal induction heating problem reads: find the magnetic field $H:\Omega\to \C$ and temperature $u:\Omega\to\R$ such that 
\begin{subequations}\label{eq:strong_solenoidal_theory}
	\begin{align}
		-\Delta H + i\omega \mu(u) H &= 0 \qquad \qquad \text{ in } \ \Omega, \label{eq:strong_magn}\\
		H & = H_\circ  \qquad \quad  \ \text{ on } \partial \Omega, \label{eq:strong_magn_bc} \\
		-\Delta u &= |\nabla H|^2 \qquad \text{ in } \ \Omega, \label{eq:strong_heat}\\
		u&=0  \qquad \qquad \ \ \text{on } \partial \Omega. \label{eq:strong_heat_bc}
	\end{align}
\end{subequations}
Here: 
\begin{itemize}
	\item $\omega$ is a positive constant representing the angular frequency of the alternating current,
	\item $H_\circ\in\C$ is a  constant representing the magnetic field strength in the surrounding air, and
	\item $\mu(\cdot)$ is a function depending on temperature representing the magnetic permeability of the material. We assume that $\mu(\cdot)$ is Lipschitz continuous, that is, that there exists $\Clip>0$ such that
	\begin{equation}\label{Eq:mu-Lip}
	|\mu(s)-\mu(t)|\le \Clip\, |s-t|\qquad\forall\, s,t\in\R\,,
	\end{equation}
	 and that there exist positive constants $\mu_\circ$ and $\mu^\circ$ such that
	\begin{equation}\label{eq:mu_is_bounded}
		0 < \mu_\circ \leq \mu(s) \leq \mu^\circ \qquad \forall \ s \in \R.
	\end{equation}
\end{itemize} 

\begin{Remark}
The above system is, in fact, a further simplification of the two-dimensional solenoidal problem presented in
\cite{Touzani-Rappaz-2014},  where the partial differential equations take the form
	\begin{align*}
		-\textrm{div}(\sigma^{-1}(u) \nabla H) + i\omega \mu(u) H &= 0 && \text{ in } \ \Omega, \\
		-\textrm{div}(\kappa (u) \nabla u) &= \frac{1}{2\sigma (u)}|\nabla H|^2 && \text{ in } \ \Omega\,.
	\end{align*}
In this work we have set the electrical conductivity $\sigma(\cdot)$ and thermal conductivity $\kappa(\cdot)$ equal to one.  This choice has been made mainly for simplicity.  In fact,  considering these coefficients
as not constant adds difficulties that are, for the most part, technical distractions. 
Similarly,  in the steady-state the boundary of the billet is subject to non-homogeneous Dirichlet conditions,
although we have opted to set the boundary to zero to avoid technical diversions. $\Box$
\end{Remark}

The main difficulty with the solenoidal induction heating problem is the source term in \eqref{eq:strong_heat}. Following standard practices, if we assume that the weak formulation of \eqref{eq:strong_magn} is imposed in $H^1$, then the source term of \eqref{eq:strong_heat} appears to belong only to $L^1$. If this is the case, \eqref{eq:strong_heat} cannot be written as a variational problem with the trial and test functions belonging only to $H^1_0$, so we cannot apply Lax-Milgram's Lemma. Additionally, since $L^1$ is neither reflexive nor it has a predual, we cannot apply the generalised version of Lax-Milgram's Lemma either \cite[Theorem~25.9]{Ern-Guermond-II-2021}.  In this paper, we resolve this issue by showing that any
possible weak solution of \eqref{eq:strong_solenoidal_theory} satisfies $|\nabla H|^2 \in L^2(\Omega)$, and thus has additional regularity that
allows the writing of a standard weak formulation.

\subsection*{Previous Literature on Well-Posedness and Finite Element Analysis}

Up to our best knowledge, the first paper that analyses the steady-state quasi-static solenoidal induction heating problem is  \cite{Clain-Touzani-1997a}, where the authors prove existence of solutions in $W^{1,p}$  for $2 < p < \infty$, under the assumption that  $\Omega$ has a $C^2$ boundary.  The analysis in \cite{Clain-Touzani-1997a} was extended to the case of unbounded material coefficients in \cite{Clain-Touzani-1997b}. 
The problem including a time-dependent heat equation was considered in \cite{Parietti-Rappaz-1998},  again under the assumption of
a smooth boundary.  In \cite{Parietti-Rappaz-1999}  the results from \cite{Parietti-Rappaz-1998} are extended to show that under even stronger regularity assumptions, there exists a unique finite element solution that converges to the continuous solution, assuming that the mesh is quasi-uniform.  
In all papers just mentioned $\mu(\cdot)$ is assumed to be a positive constant,  while the coefficients $\kappa(\cdot)$ and $\sigma(\cdot)$ are assumed to be Lipschitz continuous, and the problem is rewritten using a Kirchoff transformation. 

Other than the papers mentioned above there is, up to our best knowledge,  a lack of works dedicated to the analysis and approximation of the system \eqref{eq:strong_solenoidal_theory},  especially in terms of
finite element methods. Now, a problem related to  \eqref{eq:strong_solenoidal_theory} is the thermistor problem. This problem, which contains a real-valued version of \eqref{eq:strong_solenoidal_theory}, has been analysed  extensively, and the analysis of the continuous problem can be
found, e.g., in  \cite{Roubicek-2013}.  While there are numerous works proposing finite element methods for this problem, their convergence
analysis usually requires either a mixed finite element approach \cite{Zhu-2006,Zhu-2011}, or uniform $L^\infty$ estimates on
the discrete solution \cite{Holst2010}.   The former of these approaches either increase the computational cost,  while for the latter, proving uniform bounds on the discrete solution requires to impose that the computational mesh satisfies appropriate geometric conditions, which are especially challenging to satisfy (see, e.g., \cite{Barrenechea-John-Knobloch-2025} for an exposition on this). 
These restrictions were lifted in the work \cite{Jensen-Malqvist-2012}, where, first a change
of variable in the problem was made in order to prove a uniform bound on the continuous solution. Then,  a truncation was introduced
 in order to mimic this bound at the discrete level, and so circumvent the need to prove uniform $L^\infty$ bounds on the discrete
 solution. This approach was then extended to the time dependent case in \cite{Jensen-Malqvist-2022}.

\subsection*{Our Contribution}
Based on the above discussion, in this work our aim is twofold. First, we are interested in proposing a finite element discretisation
for \eqref{eq:strong_solenoidal_theory} that can be proven convergent to a solution of \eqref{eq:strong_solenoidal_theory} without assuming a regular, convex domain, and without the need to impose stringent conditions on the computational mesh.  As a byproduct of this
analysis, we will be able to prove existence of weak solutions to \eqref{eq:strong_solenoidal_theory} in a setting slightly more
general than the one present in previous works. 

The weak formulation of this problem has been carried out in the past by assuming that
$\Omega$ has a $C^2$ boundary.  In this work we circumvent this restriction by using a regularity
result proved in \cite{Jerison-Kenig-1995} which allows us to justify an extra regularity for
any solution of \eqref{eq:strong_magn} on a Lipschitz domain.
This will allow us to write a weak formulation for the problem that is suitable for finite element methods. Then,
to motivate our approach, we introduce a standard finite element method for the model problem.  In the process of presenting its analysis,
the need to replicate the stability for the discrete magnetic field leads to imposing two essential conditions, namely, that the domain is
convex, and that the  family of meshes is quasi-uniform. Under these conditions, we are able to prove existence of solutions, and their strong convergence, up to subsequences in the case the continuous solution is not unique.  

To avoid the restrictions mentioned above, we then propose a bound-preserving finite element method for the heat equation in
the model system. The method is the application of the one proposed in \cite{Barrenechea-2024} to this model problem, but with a fundamental variation, as there are no natural ways of obtaining upper bounds on the temperature for the induction heating problem. In fact,
the change of variables made in \cite{Jensen-Malqvist-2012} cannot be carried out in the current case, and so,  even if it can be proven fairly
easily that the temperature is positive, obtaining an upper bound for it is far from being a trivial task.  To overcome this difficulty, we
build the method as a double sequence of discrete solutions indexed by $(h,k)_{h>0, k\in\mathbb{N}}$, where $h$ is the mesh size and $k$ is an increasing sequence of upper bounds.  Then,  we refine
the results from \cite{Barrenechea-2024} to show that, for a given upper bound $k$, the finite element method converges as $h\to 0$ to the solution of a variational inequality of double obstacle type.  So, after generalising classical regularity results for the obstacle problem slightly,  we show that for $k$ large enough
this limiting function actually solves the weak formulation (and not only a variational inequality), thus proving the convergence
of the method to a weak solution of \eqref{eq:strong_solenoidal_theory}, and extending the existence of solutions to non-convex domains, as a byproduct.

The plan of the paper is as follows.  First, in Section~\ref{sec:general_setting} we start with some preliminary results, and in Section~\ref{sec:WP} we present the weak formulation, and justify why it can be written.  Section~\ref{sec:finite_element} is devoted to present and analyse the standard finite
element method for \eqref{eq:strong_solenoidal_theory}. Then, in Section~\ref{sec:BPM_theory} we present the bound-preserving method, where we first recall it for a scalar problem and prove results not previously available for it, and then we prove existence of solutions and convergence to a solution
of the continuous problem. In Section~\ref{sec:numerics} we report the results of a series of numerical experiments, and conclude the work
with some conclusions and open questions in Section~\ref{Sec:Concl}.

\section{General Setting and Preliminary Results}\label{sec:general_setting}
Let $\Omega\subseteq\R^2$ be an open, bounded,  polygonal, Lipschitz  domain.  We use standard notation for Sobolev spaces, aligned with, e.g.,  \cite{Ern-Guermond-I-2021}. In particular, 
we denote by
$\|\cdot\|_{0,p,\Omega}$ the $L^p(\Omega)$-norm, omitting the subscript $p$ when
$p=2$. For $s \geq 0$ and $p \in [1,\infty]$, we write
$\|\cdot\|_{s,p,\Omega}$ (resp.\ $|\cdot|_{s,p,\Omega}$) for the norm (resp.
seminorm) in $W^{s,p}(\Omega)$, and again drop the subscript $p$ when
$p=2$. We set
\[
W^{1,p}_0(\Omega) = \{ v \in W^{1,p}(\Omega) : v=0 \text{ on } \partial \Omega \},
\]
and denote by $W^{-1,p'}(\Omega)$ the dual of $W^{1,p}_0(\Omega)$,  where $p'$ is
the dual exponent of $p$, that is, $1/p + 1/{p'}=1$. 
We will also use intensively spaces of complex-valued functions, with the usual notations for Lebesgue and
Sobolev spaces. We will distinguish the inner product and norms as follows:
\begin{equation*}
(f,g)_{\Omega,\C}\coloneqq\int_\Omega  f(\boldsymbol{x}) \bar{g}(\boldsymbol{x}) \mathrm{d}\boldsymbol{x}
\quad\textrm{and}\quad \|f\|_{0,\Omega,\C}\coloneqq (f,f)^{1/2}_{\Omega,\C}\,,
\end{equation*}
for every $f,g\in L^2(\Omega;\C)$.  The rest of the notations follow from this convention in a natural way.
Finally, in the sequel we
do not distinguish between inner products and duality pairings for
scalar- or vector-valued functions. 

Throughout this work we will use repeatedly the following Poincar\'e's inequality
for every $1\le p \le\infty$, there exists $C_P>0$ (depending on $p$ and the diameter of $\Omega$, but we shall not track this dependence),
such that
\begin{equation}\label{Poincare}
\|v\|_{0,p,\Omega}\le C_P\,|v|_{1,p,\Omega}\qquad\forall\, v\in W^{1,p}_0(\Omega)\,.
\end{equation}
This inequality is also valid in the complex setting, by just applying it separately to the real and complex parts of the elements 
in $W^{1,p}_0(\Omega;\C)$. 
\footnotemark\footnotetext{For its proof, see,  e.g. \cite[Corollary~9.198]{Brezis-2011}  for the proof for $p<\infty$,  while the proof for $p=\infty$ follows by the Sobolev
embedding  $v \in W^{1,\infty}_0(\Omega) \hookrightarrow W^{1,t}_0(\Omega) \hookrightarrow C(\bar\Omega)$,  where $t > 2$.  As a consequence
\[
  \| v \|_{0,\infty,\Omega} \leq \Cemb \| \nabla v \|_{0,t,\Omega} \leq \Cemb |\Omega|^{1/t} \| \nabla v \|_{0,\infty,\Omega}
\]}

In addition, we recall Sobolev's embedding Theorem, Rellich-Konrachov's Theorem (see \cite[Chapter~9]{Brezis-2011} for precise statements): 
For $1\le p\le\infty$
we have that
\begin{itemize}
\item $W^{1,p}(\Omega)$ is continuously embedded in $L^q(\Omega)$, for all $q\in[1,2p/(2-p))$ if $p<2$, all
$q\in[1,\infty)$ if $p=2$,  and all $q\in [1,\infty]$ if $p>2$;
\item as a consequence, there exists $\Cemb>0$ such that, for all $p,q$ as in the previous point, the following inequality
holds
\begin{equation}\label{Ineq-Embedding}
\|v\|_{0,q,\Omega}\le \Cemb\|v\|_{1,p,\Omega}\qquad\forall\, v\in W^{1,p}(\Omega)\,;
\end{equation}
\item the above inclusions are compact.
\end{itemize}

In addition, we explicitly define the sets of functions that satisfy certain bounds a.e.  in the domain $\Omega$.  More
precisely, for $\epsilon,k>0$, we define
\begin{equation*}
	\mathcal{D}_{\epsilon,k}(\Omega) \coloneqq \{f \in \mathcal{D}(\Omega): -\epsilon \leq f(\boldsymbol{x}) \leq k \text{ a.e. in } \Omega\},
\end{equation*}
and 
\begin{equation}
\label{eq:DefOfVkeps}
	V^{\epsilon,k} \coloneqq \{v \in H^1_0(\Omega): -\epsilon  \leq v(\boldsymbol{x}) \leq k \text{ a.e. in } \Omega\}.
\end{equation}

These convex sets will be used to prove convergence of the BPM to a solution of the solenoidal induction heating problem, where it will be instrumental that  $V^{\epsilon,k}$ is closed in $H^1_0(\Omega)$. This is stated in the following result, whose proof is analogous to that of \cite[Proposition II.5.2]{Stampacchia-2000}.

\begin{lemma}\label{lem:D_k_closed_in_H1}
	The closure of $\mathcal{D}_{\epsilon,k}(\Omega)$ with respect to the $H^1_0(\Omega)$-norm is $V^{\epsilon,k}$.
\end{lemma}

As a final preliminary we recall some higher integrability properties of the Poisson problem in Lipschitz domains which we will use repeatedly. As shown in \cite[Theorem~0.5(a)]{Jerison-Kenig-1995}, there is a $P>4$ (recall that $\Omega\subset\R^2$), such that if $F \in W^{-1,p}(\Omega)$ for $p \in [P',P]$ then $\Phi$, the weak solution to
\[
  -\Delta \Phi = F \quad \text{ in } \Omega, \qquad \qquad \Phi = 0, \text{ on } \partial\Omega,
\]
satisfies
\begin{equation}
\label{eq:JKWm1p}
  \| \nabla \Phi \|_{0,p,\Omega} \leq C_{JK}(p) \| F \|_{-1,p,\Omega}, \qquad \forall p \in [P',P].
\end{equation}
In addition, if $F \in L^2(\Omega)$, we have $\Phi \in H^{3/2-\delta}(\Omega)$ for every $\delta>0$; see \cite[Theorem~B.2]{Jerison-Kenig-1995}.

%

 \section{The Weak Formulation}\label{sec:WP}

We begin by providing a weak formulation for \eqref{eq:strong_solenoidal_theory}. Namely,
find $(H,u) \in H^1(\Omega;\C) \times H^1_0(\Omega)$ such that $H = \tilde{H} + H_\circ$, $\tilde{H} \in H^1_0(\Omega;\C)$, and
\begin{subequations} \label{eq:weak_existence_form}
	\begin{align}
		(\nabla \tilde{H}, \nabla Q)_{\Omega, \C}+  \left(i \omega \mu(u)\tilde{H}, Q\right)_{\Omega, \C} &= (i \omega \mu(u)H_\circ, Q)_{\Omega, \C}, \label{eq:weak_H_coupled}\\
		(\nabla u, \nabla v)_{\Omega} &= \left(|\nabla H|^2, v\right)_{\Omega},\label{eq:weak_u_coupled}
	\end{align}
\end{subequations}
for every $(Q,v) \in H^1_0(\Omega;\C) \times H^1_0(\Omega)$.

At first sight, the right-hand side of  \eqref{eq:strong_heat} belongs only to $L^1(\Omega)$. So, the above weak formulation cannot, in
principle, be written.  For this reason in the rest of this section
we will prove results that justify why any possible solution of \eqref{eq:weak_H_coupled} is, in fact, more regular than what we might expect at first,  which will justify the choice of test functions in \eqref{eq:weak_u_coupled}.  We start with the following result concerning the regularity, and uniform boundedness, of $H$.

\begin{lemma}\label{lem:H_bounded_H1}
	Let $z \in H^1_0(\Omega)$ be arbitrary. Then, there exists a unique solution $\tilde{H}\in H^1_0(\Omega;\C)$ of the following  problem:
	\begin{equation}\label{eq:H_tilde_weak_form}
	\left(\nabla \tilde{H} , \nabla Q \right)_{\Omega, \C}  + \left( i \omega \mu(z)\tilde{H},Q\right)_{ \Omega, \C}  = \left(i\omega \mu(z)H_\circ ,Q \right)_{\Omega, \C},
\end{equation}
for all $Q \in H^1_0(\Omega;\C)$.  In addition,  there exists a positive constant $C_1$ depending only on $\Omega$, $\omega$, $\mu^\circ$, and $H_\circ$ (but, crucially, not on $z$) such that,  if we define $H = \tilde{H} + H_\circ$, then
	\begin{equation}\label{Eq:Bound-H-H1}
		\|H\|_{1,\Omega, \C } \leq C_1\,.
	\end{equation}
Moreover,  $H \in W^{1,4}(\Omega;\mathbb{C})$, and  there exists a constant $C_2$ depending only on $\Omega$, $\omega$, $\mu^\circ$, and $H_\circ$ such that
	\begin{equation}\label{Eq:Bound-H-W14}
		\|H\|_{1,4,\Omega, \C}^{} \leq C_2.
	\end{equation}
\end{lemma}

\begin{proof}
	Fix $z \in H^1_0(\Omega)$.  The well-posedness of \eqref{eq:H_tilde_weak_form} follows from Lax-Milgram's Lemma (see, e.g., 
	\cite[Lemma~25.2]{Ern-Guermond-II-2021} for its complex version). Taking $Q = \tilde{H}$ as a test function in \eqref{eq:H_tilde_weak_form} and taking real parts in both sides of the equation we get
		\begin{equation*}
		\real\left(\| \nabla \tilde{H} \|_{0, \Omega, \C}^2 + i \omega \| \sqrt{\mu(z)}\tilde{H} \|^2_{0, \Omega, \C}\right)  = \real\left(\left( i\omega\mu(z) H_\circ,\tilde{H} \right)_{\Omega, \C}\right), \\
		\end{equation*}
		which, using  the Young and Poincar\'e inequalities, implies that
		\begin{equation*}
		 \| \nabla \tilde{H} \|_{0, \Omega, \C}^2   = \real\left(\left( i\omega\mu(z) H_\circ, \tilde{H} \right)_{\Omega, \C}\right) \leq \omega \mu^\circ\|H_\circ\|_{0, \Omega, \C}\|\tilde{H}\|_{0, \Omega, \C}\leq \omega \mu^\circ C_P \|H_\circ\|_{0, \Omega, \C }\|\nabla \tilde{H}\|_{0, \Omega, \C}.
		\end{equation*}
	Dividing through by $ \| \nabla \tilde{H} \|_{0, \Omega, \C}$  and using that the seminorm and norm in $H^1_0(\Omega;\C)$ are
	equivalent we get that $
		 \|\tilde{H}\|_{1, \Omega, \C}^{} \leq C\, \omega \mu^\circ  \|H_\circ\|_{0, \Omega, \C }. $
	Finally,  since $H_\circ$ is a constant, we have that $\|H_\circ\|_{1,\Omega, \C} = \|H_\circ\|_{0,\Omega, \C}$, and so
	\begin{equation*}
		\|H\|_{1, \Omega, \C}^{} 
		 \leq \|\tilde{H} \|_{1, \Omega, \C}^{} + \| H_\circ\|_{1, \Omega, \C}^{}
		 \leq C\, \omega \mu^\circ \|H_\circ\|_{0, \Omega, \C }^{} + \| H_\circ\|_{0, \Omega, \C}^{}
		\eqqcolon C_1^{},
	\end{equation*}
	which proves  \eqref{Eq:Bound-H-H1}.
	
	To prove \eqref{Eq:Bound-H-W14} we first split equation \eqref{eq:H_tilde_weak_form} into real and complex parts. Writing $\tilde{H}  = H - H_\circ$ and $\tilde{H} = \tilde{H}_r^{} + i\tilde{H}_c^{}$, where $\tilde{H}_r^{}, \tilde{H}_c^{} \in H^1_0(\Omega)$ are real functions, and writing $H_\circ = H_{\circ, r}^{} + iH_{\circ, c}^{}$ it follows that,  in the sense of distributions,
	\begin{equation*}
		-\Delta \tilde{H}_r - \omega \mu(z)\tilde{H}_c + i(-\Delta \tilde{H}_c + \omega \mu(z)\tilde{H}_r) = -\omega \mu(z)H_{\circ, c} + i \omega \mu(z)H_{\circ, r},
	\end{equation*}
	which implies that the following system of equations is satisfied:
	\begin{align}
		-\Delta \tilde{H}_r &= \omega \mu(z)(\tilde{H}_c -H_{\circ, c}) \label{eq:real_laplace},\\
		-\Delta \tilde{H}_c &= - \omega \mu(z)(\tilde{H}_r - H_{\circ, r}). \label{eq:complex_laplace}
	\end{align}
	For a fixed $z \in H^1_0(\Omega)$ there exists a unique solution of \eqref{eq:H_tilde_weak_form}, and so there exists a unique solution of the system \eqref{eq:real_laplace},\eqref{eq:complex_laplace}. Since $\mu(\cdot) \in L^\infty(\Omega)$, it follows from \eqref{Eq:Bound-H-H1} that  
	\begin{equation}\label{eq:rhs_poisson_bounded}
		\|\omega \mu(z)(\tilde{H}_c-H_{\circ, c})\|_{0, \Omega} \leq \omega \mu^\circ (\|\tilde{H}_c\|_{0,\Omega} + \|H_{\circ, c}\|_{0,\Omega} )\leq \omega \mu^\circ (C_1 + \|H_\circ\|_{0,\Omega}),
	\end{equation}
	hence $\omega\mu(z)(\tilde{H}_c-H_{\circ, c}) \in L^2(\Omega)$. Using analogous arguments, it can be shown that $- \omega \mu(z)(\tilde{H}_r-H_{\circ, r}) \in L^2(\Omega)$. Hence \eqref{eq:real_laplace} and \eqref{eq:complex_laplace} are two real Poisson problems with right-hand sides in $L^2(\Omega)$. Consider just \eqref{eq:real_laplace}: the argument for \eqref{eq:complex_laplace} follows in an identical manner. Since $\Omega$ is Lipschitz and bounded, and \eqref{eq:real_laplace} is a Poisson equation with a right-hand side in $L^2(\Omega)$, 
  it follows from  \eqref{eq:JKWm1p} that $\tilde{H}_r\in W^{1,4}_0(\Omega)$ and
	\begin{equation*}
		\|\tilde{H}_r\|_{1, 4, \Omega}\leq C_{JK}(4) \|\omega\mu(z)(\tilde{H}_c -H_{\circ, c})\|_{-1, 4, \Omega}.
	\end{equation*}
	Now, since $L^2(\Omega) \subset W^{-1, 4}(\Omega)$ and this injection is continuous (see \cite[p.~291]{Brezis-2011}), there exists a constant $C_i$ such that for every $v \in L^2(\Omega)$, we have that $\|v\|_{-1, 4, \Omega} \leq C_i \|v\|_{0,\Omega}$. Therefore, using \eqref{eq:rhs_poisson_bounded}, we have that
	\begin{equation}
		\|\tilde{H}_r\|_{1, 4, \Omega}\leq C_{JK}(4)C_i \|\omega\mu(z)(\tilde{H}_c - H_{\circ, c})\|_{0, \Omega} \eqqcolon C_r(\Omega, \omega, \mu^{\circ}, H_\circ).
	\end{equation}
	
	An identical argument shows that there exists a positive constant $C_c$ such that
	\[
	  \|\tilde{H}_c\|_{1, 4, \Omega} \leq C_c(\Omega, \omega, \mu^\circ, H_\circ) .
	\]
	Finally, using the triangle inequality, it follows that
	\begin{equation*}
		\|H\|_{1,4, \Omega,\C} \leq \|\tilde{H}_r\|_{1,4, \Omega,\C} + \|\tilde{H}_c\|_{1,4, \Omega,\C} + \|H_\circ\|_{1,4,\Omega, \C} \leq C_r + C_c + \|H_\circ\|_{0,4,\Omega, \C} \eqqcolon C_2,
	\end{equation*}
	which proves \eqref{Eq:Bound-H-W14}, and thus finishes the proof. 
\end{proof}

This last result has as an implication that any $H$ that solves \eqref{eq:H_tilde_weak_form} (and thus any possible solution of \eqref{eq:weak_H_coupled}) is regular enough so the
right-hand side of \eqref{eq:weak_u_coupled} is well-defined.  In fact, $\nabla H\in L^4(\Omega;\C)^2$, which
justifies the following result.

\begin{lemma}\label{lem:RHS_in_L2} 
	Let $H$ be the unique solution of \eqref{eq:H_tilde_weak_form} for any fixed $z \in H^1_0(\Omega)$. Then
	$|\nabla H|^2 \in L^2(\Omega)$. 
\end{lemma}

As a consequence of the last two results,  \eqref{eq:weak_existence_form} is a valid weak formulation for \eqref{eq:strong_solenoidal_theory}.  In the next two sections,  two different finite element approximations for this problem will be presented and analysed.

\section{A Standard Conforming Finite Element Method}\label{sec:finite_element}

This section is devoted to showing that the standard finite element method applied to \eqref{eq:weak_existence_form}
converges under appropriate conditions on the mesh and the domain.  As a consequence, and since in this proof no a priori
assumption on regularity is used, this provides an 
existence proof for the system \eqref{eq:weak_existence_form}.
Let $\{\mathcal{T}_h\}_{h>0}$  be a shape-regular family of simplicial triangulations of  $\Omega$ (more assumptions on the family will be made at the appropriate results).
Associated to each $\mathcal{T}_h$ we define the finite element spaces
\begin{align}
W_h&\coloneqq\{ v_h\in C^0(\overline{\Omega}): v_h|_K\in \mathbb{P}_1(K)\, \forall\, K\in \calT_h\}\quad\textrm{and}\quad
V_h\coloneqq W_h\cap H^1_0(\Omega)\,,\label{FEM-spaces-real}\\
W_{h,\C}&\coloneqq\{ Q_h\in C^0(\overline{\Omega};\C): Q_h|_K\in \mathbb{P}_1(K;\C)\, \forall\, K\in \calT_h\}\quad\textrm{and}\quad
V_{h,\C}\coloneqq W_{h;\C}\cap H^1_0(\Omega;\C)\,.  \label{FEM-spaces-complex}
\end{align}

By $\mathcal{I}_h$ we denote the Scott-Zhang interpolation operator onto either $V_h$ or $V_{h,\C}$.
Details on the definitions and the main approximation properties of this operator can be found in, e.g.,  \cite{Ern-Guermond-I-2021} and \cite{Brenner-Scott-2008}.

The standard Galerkin approximation of \eqref{eq:weak_existence_form} reads: find $(H_h, u_h) \in W_{h,\C} \times V_h$ such that $\tilde{H}_h = H_h - H_\circ\in V_{h, \C}$, and 
\begin{subequations}\label{eq:discrete_full_weak_form}
	\begin{align}
		(\nabla \tilde{H}_h, \nabla Q_h)_{\Omega, \C}+  (i \omega \mu(u_h)\tilde{H}_h, Q_h)_{\Omega, \C} &= (i \omega \mu(u_h)H_\circ, Q_h)_{\Omega, \C}, \label{eq:galerkin_H_coupled}\\
		(\nabla u_h, \nabla v_h)_{\Omega} &= \left(|\nabla {H}_h|^2, v_h\right)_{\Omega},\label{eq:galerkin_u_coupled} 
	\end{align}
\end{subequations}
for all $(Q_h, v_h) \in V_{h, \C} \times V_h$.

%

\subsection{A Priori Bounds}

This section is devoted to proving existence of solutions to \eqref{eq:discrete_full_weak_form}. The proof is based on a fixed-point result, and so,
as a first step we will prove uniform  \emph{a priori} bounds on $H_h$ and $u_h$.
We start showing the following uniform bound on $H_h$, whose proof is identical to that of the first part of 
Lemma~\ref{lem:H_bounded_H1}.

\begin{lemma}\label{lem:H_h_bounded_in_H1}
	Let $z_h \in V_h$ be fixed. Consider the problem: find $H_h \in W_{h,\C}$ such that $\tilde{H}_h \coloneqq  H_h - H_\circ  \in V_{h,\C}$, and
	\begin{equation}\label{Prob:Hh-for-any-zh}
		(\nabla \tilde{H}_h, \nabla Q_h)_{\Omega, \C}+  (i \omega \mu(z_h)\tilde{H}_h, Q_h)_{\Omega, \C} = (i \omega \mu(z_h)H_\circ, Q_h)_{\Omega, \C},
	\end{equation}
	for all $Q_h \in V_{h,\C}$. Then $\|H_h\|_{1,\Omega, \C} \leq C_1$, where $C_1$ is the constant from Lemma \ref{lem:H_bounded_H1} that depends only on $\Omega, \omega, \mu^\circ$, and $H_\circ$.
\end{lemma}

The key step in the proof of existence of solutions is the analogue of the second part in Lemma~\ref{lem:H_bounded_H1}, namely,   the need to establish  a uniform bound on  $\nabla H_h$ in $L^4(\Omega; \C)^2$. This step will require more stringent conditions on the domain $\Omega$, and
the mesh.

\begin{lemma}\label{lem:H_bounded_in_L4}
	Assume that $\Omega$ is convex,  the family of meshes $\{\mathcal{T}_h\}_{h>0}$ is quasi-uniform, and $h$ is small enough. Let $z_h \in V_h$ be fixed, and let $H_h$ be the solution of
	\eqref{Prob:Hh-for-any-zh}.
Then, there exists a constant $C_3$ depending only on $\Omega, \omega, \mu$, and $H_\circ$ such that
	\begin{equation}\label{Eq:Unif-Bound-Hh}
		\|\nabla H_h\|_{0,4,\Omega, \C} \leq C_3.
	\end{equation}
\end{lemma}
\begin{proof}
	First, note that writing $\tilde{H}_h = \tilde{H}_{h, r} + i \tilde{H}_{h, c}$ and $H_\circ = H_{\circ, r} + i H_{\circ, c}$, \eqref{Prob:Hh-for-any-zh} can be written as 
	\begin{subequations}\label{eq:discrete_real-complex}
	\begin{align}
		(\nabla \tilde{H}_{h,r}, \nabla q_h)_\Omega - (\omega \mu(z_h)\tilde{H}_{h,c}, q_h)_\Omega &= - (\omega \mu(z_h)H_{\circ, c}, q_h)_\Omega, \label{eq:discrete_real-complex-a} \\
		(\nabla \tilde{H}_{h,c}, \nabla q_h)_\Omega + (\omega \mu(z_h)\tilde{H}_{h,r}, q_h)_\Omega &=  (\omega \mu(z_h)H_{\circ, r}, q_h)_\Omega,\label{eq:discrete_real-complex-b}
	\end{align}
\end{subequations}
	for all $q_h \in V_h$.

	Let now, for $p\in [1,\infty]$,  $R_h:W^{1,p}_0(\Omega)\to V_{h}$ be the Ritz projection  defined as
	\begin{equation}\label{def:Ritz-Proj}
	(\nabla R_h w, \nabla q_h)_{\Omega} = (\nabla w, \nabla q_h)_{\Omega}\qquad\forall\, q_h\in V_{h}\,.
	\end{equation}
	In \cite{MR4748207,Rannacher-Scott-1982} it is proven that under the assumptions that $\Omega$ is convex, the family of triangulations $\{\mathcal{T}_h\}_{h>0}$ is quasi-uniform,  and $h$ is small enough, there exists a positive constant $C_{\textrm{ritz}}$ such that for all $1 < p \leq \infty$, and all $w \in W^{1,p}_0(\Omega)$,
	\begin{equation}\label{eq:ritz_inequality}
		\|\nabla R_hw\|_{0,p,\Omega} \leq C_{\textrm{ritz}} \|\nabla w\|_{0,p,\Omega}.
	\end{equation}
	From this stability, the following discrete inf-sup condition follows: for every $p \in [P',P]$, where $P$ is the same as in \eqref{eq:JKWm1p}, there exists
	$\tilde{\alpha}>0$, independent of $h$, such that
	\begin{equation}\label{eq:ritz_inequality_2}
		\tilde{\alpha}\,\|\nabla \tilde{H}_{h,r}\|_{0,p,\Omega} \leq  \sup_{q_h \in V_h} \frac{(\nabla \tilde{H}_{h,r}, \nabla q_h)_{\Omega}}{\|\nabla q_h\|_{0, p', \Omega}}\,,
	\end{equation}
	where $p'$ is such that $1/p + 1/{p'}=1$. 
	
	As a consequence of the discrete inf-sup condition  above, using \eqref{eq:discrete_real-complex-a}, H\"{o}lder's, and  Poincar\'{e}'s inequalities, it follows that
	\begin{align*}
		\tilde{\alpha}\|\nabla \tilde{H}_{h, r}\|_{0, 4, \Omega, \C} &\leq \sup_{ q_h \in V_{h}}\frac{(\nabla \tilde{H}_{h,r}, \nabla q_h)_{\Omega}}{\|\nabla q_h\|_{0, 4/3, \Omega}} \\
		&\leq  \sup_{ q_h \in V_{h}}\frac{( \omega\mu(z_h)\tilde{H}_{h,c}, q_h)_{\Omega} - ( \omega \mu(z_h)H_{\circ,c}, q_h)_{\Omega}}{\|\nabla q_h\|_{0, 4/3, \Omega}} \\
		&\leq  \sup_{ q_h \in V_{h}}\frac{ \omega \mu^\circ C_P\big(\|\tilde{H}_{h, c}\|_{0, 4, \Omega } + \|{H}_{\circ, c}\|_{0, 4, \Omega} \big)\|\nabla q_h\|_{0, 4/3, \Omega}}{\|\nabla q_h\|_{0, 4/3, \Omega, }} \\
		&= \omega \mu^\circ C_P\big(\|\tilde{H}_{h, c}\|_{0, 4,\Omega}  + \|H_\circ\|_{0, 4, \Omega, }\big).
	\end{align*}
	Next, by the Sobolev embedding Theorem \eqref{Ineq-Embedding}, using that,  thanks to  Lemma~\ref{lem:H_h_bounded_in_H1}, $\|\tilde{H}_h\|_{1, \Omega,\C} \leq C_1$,  and since $\|\tilde{H}_{h, c}\|_{1, \Omega} \leq \|\tilde{H}_{h}\|_{1, \Omega, \C}$, it follows that
	\begin{equation}\label{eq:embedding_constant}
		\tilde{\alpha} \|\nabla \tilde{H}_{h,r}\|_{0, 4, \Omega} \leq 
		\omega \mu^\circ C_P\big( \Cemb C_1 +\|H_\circ\|_{0, 4, \Omega, \C} \big).
	\end{equation}
	It can be shown in an identical manner that 
	\begin{equation*}
		\tilde{\alpha}\|\nabla \tilde{H}_{h, c}\|_{0, 4, \Omega} \leq \omega \mu^\circ C_P\big(\Cemb C_1 +\|H_\circ\|_{0, 4, \Omega, \C}\big).
	\end{equation*}
	It follows that 
	\begin{align*}
		\|\nabla \tilde{H}_h\|_{0,4,\Omega, \C} \leq \|\nabla \tilde{H}_{h, r}\|_{0, 4, \Omega} +\|\nabla \tilde{H}_{h, c}\|_{0, 4, \Omega}\leq\frac{ 2 \omega \mu^\circ C_P\big(\Cemb C_1 +\|H_\circ\|_{0, 4, \Omega, \C}\big)}{\tilde{\alpha}} \eqqcolon C_3,
	\end{align*}
	which finishes the proof.
\end{proof}

%

Lemma \ref{lem:H_bounded_in_L4} immediately shows that $|\nabla H_h|^2$ is uniformly bounded in $L^2(\Omega)$. This allows us to prove the following uniform a priori bound for $u_h$.

\begin{lemma}\label{Lem:unif-bound-uh}
	Under the hypotheses of Lemma~\ref{lem:H_bounded_in_L4} there exists a positive constant $C_4$ independent of $h$ such that for any $H_h \in V_{h,\C}$ with $\|\nabla H_h\|_{0,4,\Omega,\C} \leq C_3$, the solution  $u_h$ to \eqref{eq:galerkin_u_coupled} satisfies
	\begin{equation}
		|u_h|_{1,\Omega} \leq C_4.
	\end{equation}
\end{lemma}

\begin{proof}
	Taking $u_h$ as a test function in \eqref{eq:galerkin_u_coupled}, using the Cauchy-Schwarz and Poincar\'e inequalities, and \eqref{Eq:Unif-Bound-Hh}, it follows that 
	\begin{equation*}
		|u_h|_{1,\Omega}^2 = (|\nabla H_h|^2, u_h)_{\Omega}\leq \||\nabla H_h|^2\|_{0, \Omega}\|u_h\|_{0, \Omega} 
		\leq C^2_3C_P |u|_{1,\Omega}\,,
		\end{equation*}
which proves the result with $C_4 \coloneqq C_3^2C_P$.
\end{proof}

%

\subsection{Existence of Solutions and Convergence }

In this section we use the uniform a priori bounds from the last section to prove existence of discrete solutions, as well as their convergence to a solution of \eqref{eq:weak_existence_form}.

\begin{theorem}\label{Theo:Exis-Galerkin}
	Under the hypotheses of Lemma~\ref{lem:H_bounded_in_L4}, there exists a solution $(H_h, u_h)$ to \eqref{eq:discrete_full_weak_form}.
\end{theorem}

\begin{proof}
	The proof is based on Brouwer's fixed point theorem (see, e.g., \cite[Theorem~1.10]{Roubicek-2013}). Consider the decoupled problem: given $\hat{w}_h \in V_h$, find $(\varphi_h, w_h) \in W_{h, \C} \times V_h$ such that $\tilde{\varphi}_h \coloneqq \varphi_h - H_\circ \in V_{h, \C}$, and
	\begin{align}
		(\nabla \tilde{\varphi}_h, \nabla Q_h)_{\Omega, \C}+  (i \omega \mu(\hat{w}_h)\tilde{\varphi}_h, Q_h)_{\Omega, \C} &= (i \omega \mu(\hat{w}_h)H_\circ, Q_h)_{\Omega, \C}, \label{eq:galerkin_H_decoupled}\\
		(\nabla w_h, \nabla v_h)_{\Omega} &= \left(|\nabla \varphi_h|^2, v_h\right)_{\Omega},\label{eq:galerkin_u_decoupled} 
	\end{align}
	for all $(Q_h, v_h) \in V_{h, \C} \times V_h$.  Since both \eqref{eq:galerkin_H_decoupled} and \eqref{eq:galerkin_u_decoupled} are well-posed problems, the following mapping is well-defined:
	\begin{equation*}
		M_h\coloneqq M_{2,h} \circ M_{1,h}:V_h \rightarrow V_h: \hat{w}_h \mapsto w_h, \qquad \text{ where } \quad M_{1,h}:\hat{w}_h \mapsto \varphi_h, \quad M_{2,h} : \varphi_h \mapsto w_h.
	\end{equation*}
	Lemma \ref{Lem:unif-bound-uh} yields the following uniform bound
	\begin{equation}\label{Eq:uni-bound-Mh}
		\|\nabla M_h(\hat{w}_h)\|_{0, \Omega}\leq C_4 \qquad \forall\  \hat{w}_h \in V_h.
	\end{equation}
	
	Let $\{\hat{w}_k\}_{k\geq 0}$ be a sequence in $V_h$. Define, for every $k\in \mathbb{N}$,  $\varphi_k \coloneqq M_{1,h}(\hat{w}_k)$, which defines another sequence $\{\varphi_k\}_{k\geq 1}$, which is uniformly bounded by a constant $C_1$ in $H^1(\Omega;\C)$ (where $C_1$ is the constant from Lemma \ref{lem:H_bounded_H1}). It follows that the sequence $\{\tilde{\varphi}_k\}_{k\geq1}$, where $\tilde{\varphi}_k \coloneqq \varphi_k - H_\circ$ is also uniformly bounded in $H^1_0(\Omega;\C)$.
Similarly,  using \eqref{Eq:uni-bound-Mh}, the sequence defined by $w_k \coloneqq M_{2,h}(\varphi_k)$  is uniformly bounded  in $H^1_0(\Omega)$.
Therefore if we show that the map $M_h$ is continuous, we can apply Brouwer's Fixed Point Theorem to prove that there exists a solution to the discrete problem \eqref{eq:discrete_full_weak_form}.
Let $k,j\in\mathbb{N}$. Then for all test functions $Q_h \in V_{h,\C}$,
	\begin{align*}
		(\nabla \tilde{\varphi}_k, \nabla Q_h)_{\Omega, \C}+  (i \omega \mu(\hat{w}_k)\tilde{\varphi}_k, Q_h)_{\Omega, \C} &= (i \omega \mu(\hat{w}_k)H_\circ, Q_h)_{\Omega, \C}, \\
		(\nabla \tilde{\varphi}_j, \nabla Q_h)_{\Omega, \C}+  (i \omega \mu(\hat{w}_j)\tilde{\varphi}_j, Q_h)_{\Omega, \C} &= (i \omega \mu(\hat{w}_j)H_\circ, Q_h)_{\Omega, \C},
	\end{align*}
	which implies that
	\begin{align*}
		(\nabla (\tilde{\varphi}_k - \tilde{\varphi}_j), \nabla Q_h)_{\Omega, \C} &=  (i \omega \mu(\hat{w}_k)H_\circ - i\omega \mu(\hat{w}_j)H_\circ, Q_h)_{\Omega, \C} \\
		&\qquad - (i \omega \mu(\hat{w}_k)\tilde{\varphi}_k - i \omega\mu(\hat{w}_j)\tilde{\varphi}_j, Q_h)_{\Omega, \C} .
	\end{align*}
	Taking $Q_h=\tilde{\varphi}_k - \tilde{\varphi}_j$ as a test function in the equality above we arrive at
	\begin{align*}
		\|\nabla (\tilde{\varphi}_k - \tilde{\varphi}_j)\|^2_{0, \Omega, \C} &= (i \omega H_\circ (\mu(\hat{w}_k) - \mu(\hat{w}_j)), \tilde{\varphi}_k - \tilde{\varphi}_j)_{\Omega, \C} \\ 
		&\quad -  (i \omega \mu(\hat{w}_k)\tilde{\varphi}_k - i \omega\mu(\hat{w}_j)\tilde{\varphi}_j, \tilde{\varphi}_k - \tilde{\varphi}_j )_{\Omega, \C}.
	\end{align*}
	Using the Cauchy-Schwarz  and Poincar\'e inequalities, and the fact that $\mu(\cdot)$ is Lipschitz, we obtain that
	\begin{align*}
		\left|(i \omega H_\circ (\mu(\hat{w}_k) - \mu(\hat{w}_j)), \tilde{\varphi}_k - \tilde{\varphi}_j)_{\Omega, \C} \right|&\leq \omega|H_\circ| \|\mu(\hat{w}_k) - \mu(\hat{w}_j)\|_{0,\Omega, \C} \|\tilde{\varphi}_k - \tilde{\varphi}_j\|_{0, \Omega, \C} \\
		& \leq \omega|H_\circ| \Clip C_P\|\hat{w}_k - \hat{w}_j\|_{0, \Omega} \|\nabla (\tilde{\varphi}_k - \tilde{\varphi}_j)\|_{0, \Omega, \C}.
	\end{align*}
	Next, a straightforward computation leads to
	\begin{align*}
		(i \omega \mu(\hat{w}_k)\tilde{\varphi}_k &- i \omega\mu(\hat{w}_j)\tilde{\varphi}_j, \tilde{\varphi}_k - \tilde{\varphi}_j )_{\Omega, \C}\\ &= (i \omega \mu(\hat{w}_k)\tilde{\varphi}_k - i \omega\mu(\hat{w}_j)\tilde{\varphi}_k + i \omega\mu(\hat{w}_j)\tilde{\varphi}_k - i \omega\mu(\hat{w}_j)\tilde{\varphi}_j, \tilde{\varphi}_k - \tilde{\varphi}_j )_{\Omega, \C} \\
		& = (i \omega (\mu(\hat{w}_k) - \mu(\hat{w}_j))\tilde{\varphi}_k, \tilde{\varphi}_k - \tilde{\varphi}_j)_{\Omega, \C} + (i \omega \mu(\hat{w}_j)(\tilde{\varphi}_k - \tilde{\varphi}_j), \tilde{\varphi}_k - \tilde{\varphi}_j )_{\Omega, \C}.
	\end{align*}
  Observe that the last term on the right hand side is purely imaginary.
	Gathering the above computations,  using the fact that $\|\nabla (\tilde{\varphi}_k - \tilde{\varphi}_j)\|_{0, \Omega,\C}$ is real,  Young's inequality, the fact that $\|\tilde{\varphi}_k\|_{1,\Omega,\C}$ is uniformly bounded,  and the embedding constant $\Cemb$ from \eqref{Ineq-Embedding}, it follows that
	\begin{align*}
		\|\nabla (\tilde{\varphi}_k &- \tilde{\varphi}_j)\|_{0, \Omega,\C}^2 = \real \left( (i \omega H_\circ (\mu(\hat{w}_k) - \mu(\hat{w}_j)), \tilde{\varphi}_k - \tilde{\varphi}_j)_{\Omega, \C} + (i \omega (\mu(\hat{w}_k) - \mu(\hat{w}_j))\tilde{\varphi}_k,\tilde{\varphi}_k - \tilde{\varphi}_j)_{\Omega, \C} \right) \\
		& \leq \omega|H_\circ| \Clip C_P\|\hat{w}_k - \hat{w}_j\|_{0, \Omega} \|\nabla (\tilde{\varphi}_k - \tilde{\varphi}_j)\|_{0, \Omega, \C} + \omega \Clip  C_PC_3\|\hat{w}_k - \hat{w}_j\|_{0, \Omega}\|\tilde{\varphi}_k - \tilde{\varphi}_j\|_{0, 4, \Omega, \C} \\
		& \leq \omega \Clip  C_P(|H_\circ| + C_3\Cemb) \|\hat{w}_k - \hat{w}_j\|_{0, \Omega} \|\nabla (\tilde{\varphi}_k - \tilde{\varphi}_j)\|_{0, \Omega, \C},
		\end{align*}
which implies that
		\begin{equation*}
		\|\nabla (\tilde{\varphi}_k - \tilde{\varphi}_j)\|_{0, \Omega,\C} \leq  \omega \Clip  C_P(|H_\circ| + C_3\Cemb) \|\hat{w}_k - \hat{w}_j\|_{0, \Omega} .
	\end{equation*}
	This proves that $M_{1,h}$ is continuous since $\varphi_k - \varphi_k  = (\tilde{\varphi}_k + H_\circ) - (\tilde{\varphi}_j + H_\circ) = \tilde{\varphi}_k - \tilde{\varphi}_j$.

	To prove that $M_{2,h}$ is continuous, let $w_j \coloneqq M_{2,h}(\varphi_j)$ and $w_k\coloneqq M_{2,h}(\varphi_k)$. Then, for all $v_h \in V_h$,
	\begin{equation*}
		(\nabla (w_k - w_j),\nabla v_h)_\Omega = (|\nabla \varphi_k|^2 - |\nabla \varphi_j|^2, v_h)_\Omega.
	\end{equation*}
	Taking $v_h=w_k - w_j$ as a test function in the above equality,  using Young's inequality, Sobolev's embedding Theorem, and Lemma~\ref{lem:H_bounded_in_L4} we arrive at
	\begin{align*}
		\|\nabla (w_k - w_j)\|^2_{0,\Omega} &= (|\nabla \varphi_k|^2 - |\nabla \varphi_j|^2, w_k - w_j)_\Omega \\
		&= (\nabla \varphi_k\cdot \overline{\nabla \varphi_k} - \nabla \varphi_k\cdot \overline{\nabla \varphi_j} + \nabla \varphi_k\cdot \overline{\nabla \varphi_j} - \nabla \varphi_j\cdot \overline{\nabla \varphi_j}, w_k - w_j)_\Omega \\
		&= (\nabla \varphi_k\cdot (\overline{\nabla \varphi_k} -  \overline{\nabla \varphi_j}),  w_k - w_j)_\Omega  +( (\nabla \varphi_k- \nabla \varphi_j)\cdot \overline{\nabla \varphi_j}, w_k - w_j)_\Omega \\
		&\leq \|\nabla \varphi_k\|_{0,4,\Omega,\C} \|\overline{\nabla \varphi_k} -  \overline{\nabla \varphi_j}\|_{0, \Omega,\C}\|w_k - w_j\|_{0,4,\Omega, \C} \\ & \qquad + \|\nabla \varphi_k- \nabla \varphi_j\|_{0,\Omega,\C}\|\overline{\nabla \varphi_j}\|_{0, 4, \Omega, \C}\|w_k - w_j\|_{0,4,\Omega, \C} \\
		& \leq 2C_3\Cemb\|\varphi_k - \varphi_j\|_{1,\Omega, \C}\|\nabla(w_k - w_j)\|_{0,\Omega}\,,
		\end{align*}
		which proves that  $M_{2,h}$ is continuous as well.  The continuity of $M_h$ and \eqref{Eq:uni-bound-Mh} allow us to use 
		Brouwer's fixed point Theorem to prove  that the mapping $M_h$ has a fixed point   $u_h$.  Defining $H_h=M_{1,h}(u_h)$, then  $(H_h, u_h)$.
	 solves \eqref{eq:discrete_full_weak_form}.
\end{proof}

\begin{Remark}
It is worth mentioning that the proof of the previous result actually shows that the mapping $M_h$ is Lipschitz. As a consequence, if the constants in the previous results are tracked carefully,  then it would be possible to prove uniqueness of solutions via contraction if the coefficients of the equation are small
enough, and/or $\mu(\cdot)$ is close enough to be a constant function.
\end{Remark}

The next result shows that, up to a subsequence, the discrete solution converges to a solution of \eqref{eq:weak_existence_form}, thus
proving existence of solutions of \eqref{eq:weak_existence_form} in the case when $\Omega$ is convex.

\begin{theorem}\label{Theo:Conv-Galerkin}
	Assume the hypotheses of Lemma~\ref{lem:H_bounded_in_L4}, and let $\{(H_h, u_h)\}_{h>0}$ be a family of solutions of \eqref{eq:discrete_full_weak_form}. Then, up to a subsequence, $H_h \to H$ strongly in $H^1(\Omega ;\C)$ and $u_h \to u $ strongly in $H^1_0(\Omega)$, where $(H,u)$ solves \eqref{eq:weak_existence_form}.
\end{theorem}

\begin{proof}
	Throughout the proof, every time we mention convergence, this will
	be understood as ``convergence up to a subsequence''.  Since the family $\{(H_h, u_h)\}_{h>0}$  of discrete solutions  is uniformly bounded in $H^1(\Omega;\C)\times H^1_0(\Omega)$,
	with $\{\nabla H_h\}_{h>0}$ also uniformly bounded in $L^4(\Omega;\C)^2$, up to a subsequence, we have that $H_h \rightharpoonup H$ weakly in $H^1(\Omega; \C)$,  $ H_h \rightharpoonup  H$ weakly in $W^{1,4}(\Omega;\C)$,  and $u_h \rightharpoonup u$ weakly in $H^1_0(\Omega)$. In addition,  thanks to the Rellich-Kondrachov Theorem, $H_h\to H$ strongly  in $L^p(\Omega;\C)$  and $u_h\to u$ strongly in $L^p(\Omega)$
	for $p< \infty$.
\vspace{.15cm}	
	
\noindent\underline{Identifying the limit in \eqref{eq:galerkin_H_coupled} :}
	Let $Q \in \mathcal{D}(\Omega;\C)$ and let $Q_h \coloneqq \mathcal{I}_h(Q) \in V_{h,\C}$. Since $Q_h \rightarrow Q$ strongly in $W^{1, \infty}_{0}(\Omega;\C)$ as $h \rightarrow 0$, and recalling that $\tilde{H}_h = H_h - H_\circ$ and $\tilde{H} = H - H_\circ$,
	it follows that
	\begin{equation*}
		(\nabla \tilde{H}_h, \nabla Q_h)_{\Omega,\C} = (\nabla \tilde{H}_h, \nabla (Q_h - Q) )_{\Omega,\C} + (\nabla \tilde{H}_h, \nabla  Q )_{\Omega,\C}\rightarrow (\nabla \tilde{H}, \nabla  Q )_{\Omega,\C}, \quad\text{ as } h \rightarrow 0.
	\end{equation*}
Next,  using that $\mu(\cdot)$ is Lipschitz and  $u_h \rightarrow u$ strongly in $L^p(\Omega)$ for all $p<\infty$ we
get that $\mu(u_h) \rightarrow \mu(u)$ a.e. in $\Omega$.  This last fact, together with  $H_h\to H$ strongly in $L^2(\Omega;\C)$, leads to
\[
  (i \omega \mu(u_h)\tilde{H}_h, Q_h)_{\Omega, \C} = i \omega((\mu(u_h)-\mu(u))\tilde{H}_h, Q_h)_{\Omega,\C}+
		i \omega(\mu(u)\tilde{H}_h,Q_h)_{\Omega,\C}.
\]
Since
\[
  \left| ((\mu(u_h)-\mu(u))\tilde{H}_h, Q_h)_{\Omega,\C} \right| \leq \Clip  \| u_h - u \|_{0,4,\Omega} \| \tilde{H}_h \|_{0,\Omega,\C} \| Q_h \|_{0,4,\Omega,\C} \to 0,
\]
and
	\begin{equation*}
		(i \omega \mu(u_h)H_\circ, Q_h)_{\Omega, \C} \rightarrow (i \omega \mu(u)H_\circ,Q)_{\Omega, \C}.
	\end{equation*}
we may conclude that
\[
  (i \omega \mu(u_h)\tilde{H}_h, Q_h)_{\Omega, \C} \to (i \omega \mu(u)H , Q)_{\Omega, \C}.
\]

	This shows that $(H, u)$ satisfies \eqref{eq:weak_H_coupled}.
	
\noindent\underline{Strong convergence of $H_h$ in $H^1(\Omega;\C)$:}	
	Next, to prove the strong convergence of $H_h$ to $H$, we consider, for each $h$ in the converging subsequence,  $\mathcal{H}_h\coloneqq\mathcal{\tilde{H}}_h+H_\circ$ where $\mathcal{\tilde{H}}_h\in H^1_0(\Omega;\C)$ satisfies
	\begin{equation}\label{Eq:interm-PDe-H}
	(\nabla \mathcal{\tilde{H}}_h,  \nabla Q)_{\Omega,\C} + i\omega (\mu(u_h)\mathcal{\tilde{H}}_h,Q)_{\Omega,\C}= 
	i\omega (\mu(u_h) H_\circ, Q)_{\Omega,\C}
	\qquad\forall\, Q\in  H^1_0(\Omega;\C)\,.
	\end{equation}
	We note that $\mathcal{H}_h$ is not a finite element function,  we just use the subscript $h$ to stress that it varies with $h$ via $u_h$. 
	
	Subtracting \eqref{Eq:interm-PDe-H} from \eqref{eq:weak_H_coupled} and rearranging terms we get 
	\begin{equation}\label{Error-eq-H}
	\big(\nabla (\tilde{H}-  \mathcal{\tilde{H}}_h),  \nabla Q\big)_{\Omega,\C}
	+ i\omega  \big(\mu(u_h)(\tilde{H}-  \mathcal{\tilde{H}}_h),Q \big)_{\Omega,\C}=
	i\omega  \big((\mu(u) -\mu(u_h))(H_\circ - \tilde{H}), Q \big)_{\Omega,\C}
	\,,
	\end{equation}
	for all $Q\in  H^1_0(\Omega;\C)$.
	Taking $Q=\tilde{H}-  \mathcal{\tilde{H}}_h\in H^1_0(\Omega;\C)$ as test function in the equation above  we get to
	\[
	  	\| \nabla (\tilde{H}-  \mathcal{\tilde{H}}_h)\|_{0,\Omega}^2
	+ i\omega  \big(\mu(u_h)(\tilde{H}-  \mathcal{\tilde{H}}_h), \tilde{H}-  \mathcal{\tilde{H}}_h \big)_{\Omega,\C}=
	i\omega  \big((\mu(u) -\mu(u_h))(H_\circ - \tilde{H}), \tilde{H}-  \mathcal{\tilde{H}}_h \big)_{\Omega,\C}
	\,.
	\]
	Upon noticing that the last term on the left hand side is purely imaginary, we take the real part of this identity and use
	Young's inequality,  followed by the Lipschitz continuity
	of $\mu(\cdot)$, Poincar\'e inequality, and Sobolev embedding to arrive at
  \begin{equation} \label{intermediate-H-conver}
    \begin{aligned}
      \|\nabla (\tilde{H}-  \mathcal{\tilde{H}}_h)\|_{0,\Omega,\C}^2 &\leq 
      \omega \Clip  C_P \| u - u_h \|_{0,4,\Omega} \| \tilde{H} - H_\circ \|_{0,4,\Omega,\C} \| \nabla( \tilde{H} - \mathcal{H}_h ) \|_{0,\Omega,\C}
      \\
      &\leq \omega \Clip C_P \Cemb \| u - u_h \|_{0,4,\Omega} \| \tilde{H} - H_\circ \|_{1,\Omega,\C} \| \nabla( \tilde{H} - \mathcal{H}_h ) \|_{0,\Omega,\C},
    \end{aligned}
  \end{equation}
	from which the strong convergence of $ \mathcal{\tilde{H}}_h$ to $\tilde{H}$ in $H^1_0(\Omega;\C)$ follows from the 
	strong convergence of $u_h$ to $u$ in $L^p(\Omega)$ for all $p<\infty$.
	
	Next, taking $Q=\tilde{H}_h-  \mathcal{I}_h(\mathcal{\tilde{H}}_h)\in V_{h,\C}$ as test function in \eqref{Eq:interm-PDe-H} and \eqref{eq:galerkin_H_coupled}, and subtracting the resulting identities, we get to 
	\begin{equation*}
	(\nabla (\tilde{H}_h-\mathcal{\tilde{H}}_h), \nabla (\tilde{H}_h- \mathcal{I}_h(\mathcal{\tilde{H}}_h)))_{\Omega,\C}
	+ i\omega (\mu(u_h)(\tilde{H}_h-\mathcal{\tilde{H}}_h), \tilde{H}_h- \mathcal{I}_h(\mathcal{\tilde{H}}_h))_{\Omega,\C} =
	0\,,
	\end{equation*}
	which, after adding and subtracting $\mathcal{\tilde{H}}_h$ and rearranging terms becomes
	\begin{align*}
	&\|\nabla (\tilde{H}_h-\mathcal{\tilde{H}}_h)\|_{0,\Omega,\C}^2 + i\omega (\mu(u_h)(\tilde{H}_h-\mathcal{\tilde{H}}_h), \tilde{H}_h- \mathcal{\tilde{H}}_h)_{\Omega,\C} \\
	=&\,- (\nabla (\tilde{H}_h-\mathcal{\tilde{H}}_h), \nabla (\mathcal{\tilde{H}}_h- \mathcal{I}_h(\mathcal{\tilde{H}}_h)))_{\Omega,\C}
	+ i\omega( (\mu(u_h)(\tilde{H}_h-\mathcal{\tilde{H}}_h), \mathcal{\tilde{H}}_h- \mathcal{I}_h(\mathcal{\tilde{H}}_h))_{\Omega,\C}\,.
	\end{align*}
	We note, once again, that the last term on the left hand side is purely imaginary.
	Taking the real part of the equality above, and using on its right-hand side Young's inequality, the fact that $\mu(\cdot)X$ is bounded
	uniformly above by $\mu^\circ$,  and the Poincar\'e inequality we arrive at
	\begin{equation}\label{intermediate-H-conver-2}
	\|\nabla (\tilde{H}_h-\mathcal{\tilde{H}}_h)\|_{0,\Omega,\C}^2\le  (1+\omega\mu^\circ C_P^2)\, 
	\|\nabla (\tilde{H}_h-\mathcal{\tilde{H}}_h)\|_{0,\Omega,\C}\| \nabla (\mathcal{\tilde{H}}_h- \mathcal{I}_h(\mathcal{\tilde{H}}_h))\|_{0,\Omega,\C}\,.
	\end{equation}
	The strong convergence of $\tilde{H}_h$ to $\mathcal{\tilde{H}}_h$ follows from the fact that 
	$\| \nabla (\mathcal{\tilde{H}}_h- \mathcal{I}_h(\mathcal{\tilde{H}}_h))\|_{0,\Omega}\to 0$ as $h\to 0$, which can be easily obtained by invoking the stability of the Scott-Zhang interpolant $\mathcal{I}_h$, and the fact that $\mathcal{\tilde{H}}_h$ converges strongly to $\tilde{H}$ in $H^1_0(\Omega;\C)$.

	Hence, the strong convergence of $\tilde{H}_h$ to $\tilde{H}$ in $ H^1_0(\Omega;\C)$ follows collecting \eqref{intermediate-H-conver}
	and \eqref{intermediate-H-conver-2}, and using the triangle inequality.

\vspace{.15cm}
\noindent\underline{Identifying the limit in \eqref{eq:galerkin_u_coupled} :}	
Now, we move onto showing that $(H,u)$ solves \eqref{eq:weak_u_coupled}.  Let $v \in \mathcal{D}(\Omega)$, and let $v_h \coloneqq \mathcal{I}_h(v)$,  so that $v_h \rightarrow v$ strongly in $H^{1}_0(\Omega)$.  First,  since $u_h$ converges weakly to $u$, and is  uniformly bounded in $H^1_0(\Omega)$, we easily see that
	\begin{equation*}
		(\nabla u_h, \nabla v_h)_\Omega = (\nabla u_h, \nabla v)_\Omega+ (\nabla u_h, \nabla (v_h-v))_\Omega \rightarrow (\nabla u, \nabla v)_\Omega \quad\text{ as } h \rightarrow 0.
	\end{equation*}
	In addition, since $H_h\to H$ strongly in $H^1(\Omega;\C)$, then $|\nabla H_h|^2\to |\nabla H|^2$ strongly
in $L^1(\Omega)$,  we have that
\begin{align*}
\big|(|\nabla H_h|^2,v_h)_\Omega - ( |\nabla H|^2,v)_\Omega\big|&=\,\big |(|\nabla H_h|^2-|\nabla H|^2,v_h)_\Omega + ( |\nabla H|^2,v_h-v)_\Omega\big|\\
&\le\,  \||\nabla H_h|^2-|\nabla H|^2\|_{0,1,\Omega}\|v_h\|_{0,\infty,\Omega} +
\||\nabla H|^2\|_{0,1,\Omega}\|v_h-v\|_{0,\infty,\Omega}\\
&\rightarrow \, 0\,,
\end{align*}
after using that $v_h\to v$ in $L^\infty(\Omega)$ as $h\to 0$. 

As a conclusion, $(H,u)$ also satisfies \eqref{eq:galerkin_u_coupled}. 
\vspace{.15cm}
	
	\noindent\underline{Strong convergence of $u_h$ in $H^1_0(\Omega)$:}	
  We start by setting $\mathcal{I}_h(u)- u_h$ as a test function in \eqref{eq:weak_u_coupled} and \eqref{eq:galerkin_u_coupled}, and subtracting the results to get
	\begin{equation*}
	(\nabla (u-u_h), \nabla (\mathcal{I}_h(u)-u_h))_\Omega = (|\nabla H|^2-|\nabla H_h|^2, \mathcal{I}_h(u)-u_h)_\Omega \,,
  \end{equation*}
  which, after rearranging, yields
\begin{align*}
	\|\nabla (u-u_h)\|^2_{0,\Omega} &= 
    -(\nabla (u-u_h), \nabla(\mathcal{I}_h(u)-u))_\Omega+ (|\nabla H|^2-|\nabla H_h|^2, \mathcal{I}_h(u)-u_h)_\Omega\\
  &\leq \|\nabla (u-u_h)\|_{0,\Omega}\|\nabla(\mathcal{I}_h(u)-u)\|_{0,\Omega} +
  \left( (|\nabla H| - |\nabla H_h|)(|\nabla H| + |\nabla H_h|), \mathcal{I}_h(u)-u_h \right)_\Omega 
  \\
  &\leq
	\frac{1}{4}\|\nabla (u-u_h)\|_{0,\Omega}^2 + \|\nabla(\mathcal{I}_h(u)-u)\|_{0,\Omega}^2
	\\
	&+
  \| \nabla (H - H_h) \|_{0,\Omega,\C} \| |\nabla H|+|\nabla H_h| \|_{0,4,\Omega} \| \mathcal{I}_h(u)-u_h \|_{0,4,\Omega}.
\end{align*}
  Denote
  \[
    \mathfrak{R} \coloneqq \| \nabla (H - H_h) \|_{0,\Omega,\C} \| |\nabla H|+|\nabla H_h| \|_{0,4,\Omega} \| \mathcal{I}_h(u)-u_h \|_{0,4,\Omega}.
  \]
  Next we recall that $H \in W^{1,4}(\Omega;\C)$, and that the sequence $\{H_h\}_{h>0}$ converges weakly in $W^{1,4}(\Omega;\C)$, and so it is uniformly bounded. With this knowledge at hand, the last term on the right hand side of the previous estimate can be bounded, using the Sobolev embedding \eqref{Ineq-Embedding}, by
  \begin{align*}
    \mathfrak{R}
    &\leq C \| \nabla (H - H_h) \|_{0,\Omega,\C} \left( \| \nabla ( \mathcal{I}_h(u) - u ) \|_{0,\Omega} + \| \nabla (u - u_h) \|_{0,\Omega} \right)
    \\
    &\leq 
    C_\alpha \| \nabla (H - H_h) \|_{0,\Omega,\C} \| \nabla ( \mathcal{I}_h(u) - u ) \|_{0,\Omega}
    + C_\beta \| \nabla (H - H_h) \|_{0,\Omega,\C}^2 + \frac14 \| \nabla (u - u_h) \|_{0,\Omega}^2
    \\
    &\leq C_\gamma \left( \| \nabla (H - H_h) \|_{0,\Omega,\C}^2 + \| \nabla(\mathcal{I}_h(u) - u ) \|_{0,\Omega}^2 \right)
    + \frac14 \| \nabla (u - u_h) \|_{0,\Omega}^2,
  \end{align*}
  for some constants $C_\alpha, C_\beta, C_\gamma >0$ that are independent of $h$. In conclusion, for another constant $C_\upsilon$, that is also independent of $h$, we have
  \[
    \frac12 \|\nabla (u-u_h)\|^2_{0,\Omega} \leq C_\upsilon \left(
      \| \nabla ( \mathcal{I}_h(u) - u ) \|_{0,\Omega}^2 + \| \nabla (H - H_h) \|_{0,\Omega,\C}^2
    \right) \, .
  \]
  The convergence of $u_h$ to $u$ in $H^1_0(\Omega)$ now follows from the convergence $H_h \to H$ in $H^1_0(\Omega;\C)$ and $\mathcal{I}_h(u) \to u$ in $H^1_0(\Omega)$.
This finishes the proof.
\end{proof}

Throughout this section there have been three main assumptions, namely, that $\Omega$ is convex,  the family of triangulations is quasi-uniform, and that $h$ is \emph{small enough}.  
Under these conditions we have proven convergence of the finite element solution to a solution of
the continuous problem, thus also proving existence of weak solutions.  The question of whether these hypotheses can be removed in
the proof of the discrete inf-sup condition \eqref{eq:ritz_inequality_2} remains, up to our best knowledge, mostly open. In this regard, for instance, we mention \cite{MR2869035}, where stability of the Ritz projection over ``mildly'' graded meshes is obtained. To our knowledge, however, no reference addresses the case of a non-convex domain.  In the next
section we will present a new finite element method for \eqref{eq:weak_existence_form} for which we will be able to prove
the same results as those proven in this section, but  with less stringent
assumptions on the domain and the mesh.

%

\section{The Bound-Preserving Method}\label{sec:BPM_theory}

In this section we will present a finite element method that converges for non-convex domains and general shape-regular
triangulations, thus removing the assumptions made in the last section.  As a first step into the discussion, we will
review the bound-preserving method (BPM) proposed in \cite{Barrenechea-2024}, and prove an extension of the results known for a scalar
equation that will be instrumental in our proof of convergence of the coupled problem. 

\subsection{The Scalar Problem}

We consider  the following problem: given $f\in L^2(\Omega)$, find $z\in H^1_0(\Omega)$ such that
\begin{equation}\label{Laplacian}
(\nabla z,\nabla v)_\Omega = (f,v)_\Omega\qquad\forall v\in H^1_0(\Omega)\,.
\end{equation}
We now describe the BPM, proposed in \cite{Barrenechea-2024},  when applied to \eqref{Laplacian}, and extend the convergence results known for this method.
We let $\epsilon,k>0$. 
Define
\begin{equation}\label{Def:Vk}
  V_h^{\epsilon,k} \coloneqq \{ v_h\in V_h: v_h(\boldsymbol{x}_i)\in [-\epsilon, k]\;\textrm{for all}\; i=1,...,n\}\,,
\end{equation}
where $\{ \boldsymbol{x}_i \}_{i=1}^n$ are the internal nodes of $\mathcal{T}_h$.  Notice that $V_h^{\epsilon,k}$ is a closed convex subset of $V_h$.

\begin{Remark}
As we will mention in more detail in the next section, thanks to the maximum principle, the temperature
$u$ is non-negative in the domain $\Omega$.  This hints at the possibility of considering $\epsilon=0$
but, for technical reasons,  in the proof of the main convergence theorem below we will require $\epsilon>0$,
although it can be arbitrarily small.  So, we will assume that $\epsilon \ll k$, and in practical computations
$\epsilon$ will be taken to be near the machine precision. 
\end{Remark}

Denote the Lagrange basis functions of $V_h$ by $\{\phi_i\}_{i=1}^n$. Given $v_h\in V_h$ we split it as 
$v_h=v_h^{(\epsilon,k)+}+v_h^{{(\epsilon,k)-}}$, where
\begin{equation}\label{k-trunc-def}
v_h^{(\epsilon,k)+} \coloneqq \sum_{i=1}^n \max \{-\epsilon, \min \{ v_h(\boldsymbol{x}_i), k\} \}\phi_i, \qquad\textrm{and}\; v_h^{{(\epsilon,k)-}}\coloneqq v_h-v_h^{(\epsilon,k)+}\,.
\end{equation}

With these notations, the BPM proposed in \cite{Barrenechea-2024} reads: find $z_h\in V_h$ such that
\begin{equation}\label{BPM-scalar}
(\nabla z_h^{(\epsilon,k)+},\nabla v_h)_\Omega+ s(z_h^{(\epsilon,k)-},v_h)= (f,v_h)_\Omega \qquad\forall\, v_h\in V_h\,,
\end{equation}
where, the stabilising bilinear form $s$ is given by
\begin{equation}\label{s-definition}
	s(v_h, w_h) \coloneqq \alpha \sum_{i=1}^n v_h(\boldsymbol{x}_i)w_h(\boldsymbol{x}_i).
\end{equation}
Here $\alpha$ is a non-dimensional constant.  In \cite{Barrenechea-2024} the method \eqref{BPM-scalar} has been proven to be well-posed and optimally convergent,  in the sense that it can be proved that $z_h^{(\epsilon,k)+}$ is the
best approximation of $z$ in $V_h^{\epsilon,k}$, from which convergence can be deduced provided $z\in V^{\epsilon,k}$, and rates can be deduced if $z$ is smooth.
However, one question that has not been answered so far with respect to the BPM is the following: what if $z$, the solution to \eqref{Laplacian}, does not belong to $ V^{\epsilon,k}$?
In other words,  does the method \eqref{BPM-scalar} converge when the bounds imposed are not the physical ones? And if so,  to what?

In order to answer the question just posed, we recall a fundamental aspect of the analysis of \eqref{BPM-scalar}, namely,  that it  can be written in an equivalent way as a variational
inequality. More precisely, if $z_h$ solves \eqref{BPM-scalar}, then its constrained part,
$z_h^{(\epsilon,k)+}\in V_h^{\epsilon,k}$, solves the following variational inequality:
\begin{equation}\label{VI-scalar}
(\nabla z_h^{(\epsilon,k)+}, \nabla(v_h - z_h^{(\epsilon,k)+}))_\Omega \geq (f, v_h - z_h^{(\epsilon,k)+})_\Omega \qquad \forall\ v_h \in V_h^{\epsilon,k}.
\end{equation}
Reciprocally, if $z_h^{(\epsilon,k)+}\in V_h^{\epsilon,k}$ satisfies \eqref{VI-scalar}, then there
exists a unique $z_h^{(\epsilon,k)-}\in V_h$ such that $z_h \coloneqq z_h^{(\epsilon,k)+}+z_h^{(\epsilon,k)-}$
solves \eqref{BPM-scalar}.

One first remark that will be used in the results below is the following: taking $v_h=0\in V_h^{\epsilon,k}$ as test function in \eqref{VI-scalar}, rearranging terms, and using
the Cauchy-Schwarz and Poincar\'e inequalities  leads to the following uniform bound:
\begin{equation}\label{uhk-bound}
|z_h^{(\epsilon,k)+}|_{1,\Omega}\le C_P\,\|f\|_{0,\Omega}\qquad\forall \ h>0\,.
\end{equation}

In addition, the following stability and convergence result holds. Its proof is almost standard \cite{MR448949,MR391502,MR3393323} in the context of convergence of finite element methods for variational inequalities, but with a small twist that makes it suitable for the current setting. Namely, we allow the right hand side to change, and we assume that it converges only in $L^1(\Omega)$.

\begin{lemma}\label{lem:lem-non-physical}
	Let $f \in L^2(\Omega)$ and  $\{f_h\}_{h>0}\subset L^1(\Omega)$ be such that, as $h\to 0$, we have $ f_h\rightarrow f$ strongly in $L^1(\Omega)$.
	Let also $\mathfrak{z}_h \in V^{\epsilon,k}_h$ be the unique solution of
	\begin{equation}\label{VI-scalar-discrete}
	(\nabla \mathfrak{z}_h  \nabla(v_h - \mathfrak{z}_h))_\Omega \geq (f_h, v_h - \mathfrak{z}_h)_\Omega\qquad  \forall\ v_h \in V_h^{\epsilon,k},
	\end{equation}
	and $\mathfrak{z}\in V^{\epsilon,k}$ the unique solution  of the variational inequality
	\begin{equation}\label{VI-scalar-cont}
		(\nabla \mathfrak{z},  \nabla(v - \mathfrak{z}))_\Omega \geq (f, v - \mathfrak{z})_\Omega\qquad  \forall\ v \in V^{\epsilon,k}.
	\end{equation}
	Then, 
	$|\mathfrak{z}_h - \mathfrak{z}|_{1, \Omega} \rightarrow 0$ as $h \rightarrow 0$.
\end{lemma}

\begin{proof}
	First note that \eqref{VI-scalar-discrete} and \eqref{VI-scalar-cont} can be written as follows:
	\begin{align*}
		-(\nabla \mathfrak{z}_h, \nabla(v_h - \mathfrak{z}_h))_\Omega &\leq -(f_h, v_h - \mathfrak{z}_h)_\Omega \qquad \forall \ v_h \in V_h^{\epsilon,k}, \\
		-(\nabla \mathfrak{z},\nabla(v - \mathfrak{z}))_\Omega &\leq -(f, v - \mathfrak{z})_\Omega  \qquad\forall \ v \in V^{\epsilon,k}.
	\end{align*}
	
	Using that $\mathcal{D}_{\epsilon,k}(\Omega)$ is dense in $V^{\epsilon,k}$, there exists a sequence $\{\phi_\delta\}_{\delta >0}\subset \mathcal{D}_{\epsilon,k}(\Omega)$ such
	that $\lim_{\delta\to 0}\|\phi_\delta-\mathfrak{z}\|_{1,\Omega}=0$.  Let  $i_h:C^0(\overline{\Omega})\to V_h$ be the Lagrange interpolation operator
	(see, e.g., \cite{Raviart-Thomas-1992}). Then,  for every $\delta>0$, we have that $i_h(\phi_\delta)\in V_h^{\epsilon,k}$.
	In addition, $\mathfrak{z}_h\in V_h^{\epsilon,k} \subset V^{\epsilon,k}$. Then, using the above rewriting of \eqref{VI-scalar-discrete} and \eqref{VI-scalar-cont} we get
	\begin{align*}
		0 &\leq (\nabla(\mathfrak{z} - \mathfrak{z}_h), \nabla(\mathfrak{z} - \mathfrak{z}_h))_\Omega \\
		&= (\nabla \mathfrak{z} , \nabla(\mathfrak{z} - \mathfrak{z}_h))_\Omega - (\nabla \mathfrak{z}_h, \nabla(\mathfrak{z} - \mathfrak{z}_h))_\Omega \\
		&= (\nabla \mathfrak{z} , \nabla(\mathfrak{z} - \mathfrak{z}_h))_\Omega - (\nabla \mathfrak{z}_h, \nabla(\mathfrak{z} - \phi_\delta +\phi_\delta - i_h(\phi_\delta)+i_h(\phi_\delta) - \mathfrak{z}_h ))_\Omega\\
		&= -  (\nabla \mathfrak{z} , \nabla(\mathfrak{z}_h - \mathfrak{z}))_\Omega -  (\nabla \mathfrak{z}_h , \nabla( i_h(\phi_\delta) - \mathfrak{z}_h))_\Omega-
		(\nabla \mathfrak{z}_h, \nabla(\mathfrak{z} - \phi_\delta +\phi_\delta - i_h(\phi_\delta))_\Omega\\
		&\leq -(f , \mathfrak{z}_h - \mathfrak{z})_\Omega - (f_h,   i_h(\phi_\delta) - \mathfrak{z}_h)_\Omega -  (\nabla \mathfrak{z}_h, \nabla(\mathfrak{z} - \phi_\delta))_\Omega
		 +(\nabla \mathfrak{z}_h,\nabla (\phi_\delta - i_h(\phi_\delta))_\Omega\,.
	\end{align*}

	Next, we rewrite the first two terms above as follows
	\begin{align*}
		-(f , \mathfrak{z}_h - \mathfrak{z})_\Omega &- (f_h,  i_h(\phi_\delta) - \mathfrak{z}_h)_\Omega  = -(f , \mathfrak{z}_h )_\Omega + (f , \mathfrak{z})_\Omega - (f_h,  i_h(\phi_\delta) )_\Omega + (f_h, \mathfrak{z}_h)_\Omega \\
		& = (f_h - f , \mathfrak{z}_h )_\Omega +  (f - f_h, i_h(\phi_\delta))_\Omega+		
		(f, \mathfrak{z} -\phi_\delta)_\Omega+ (f,\phi_\delta- i_h(\phi_\delta) )_\Omega\\
		&= \sum_{i=1}^4 \mathrm{T}_i.
	\end{align*}
	
	Since $\mathfrak{z}_h\in V_h^{\epsilon,k}$ we have that $\|\mathfrak{z}_h\|_{0,\infty,\Omega}\le k$ (recall that $0<\epsilon\ll k$), and so thanks to the
	strong convergence of $f_h$ to $f$ in $L^1(\Omega)$ we get
	\begin{equation*}
	\mathrm{T}_1 \le \|f_h-f\|_{0,1,\Omega}\|\mathfrak{z}_h\|_{0,\infty,\Omega}\le k\,  \|f_h-f\|_{0,1,\Omega}\to 0\quad\textrm{as}\; h\to 0\,.
	\end{equation*}
	In a completely analogous way we see that $\mathrm{T}_2\to 0$ as $h\to 0$ as well.
	The fact that $\mathrm{T}_4\to 0$ follows from the fact that $\lim_{h\to 0}\|i_h(\phi_\delta)-\phi_\delta\|_{0,\Omega}=0$.
	
	Now,  since $\mathfrak{z}_h\in V_h^{\epsilon,k}$, we take $v_h=0$ as test function in \eqref{VI-scalar-discrete}, and use that
	$\{ f_h\}_{h>0}$ is uniformly bounded in $L^1(\Omega)$ to get the following
	uniform bound for $\mathfrak{z}_h$
	\begin{equation}\label{Eq:zh-bound-W11}
	\|\nabla\mathfrak{z}_h\|_{0,\Omega}^2\le (f_h,\mathfrak{z}_h)_\Omega\le\| f_h\|_{0,1,\Omega}\|\mathfrak{z}_h\|_{0,\infty,\Omega}
	\le Ck\,.
	\end{equation}

	Gathering the above results
	 and using \eqref{Eq:zh-bound-W11} we get to
	\begin{align*}
		0 &\leq\, \lim_{h \rightarrow 0}|\mathfrak{z} - \mathfrak{z}_h|^2_{1, \Omega} \\
		&\leq\, \lim_{h \rightarrow 0} \Big( \sum_{i=1}^4 \mathrm{T}_i + \|\nabla \mathfrak{z}_h\|_{0,\Omega}\|\nabla(\mathfrak{z} - \phi_\delta)\|_{0,\Omega}+
		\|\nabla \mathfrak{z}_h\|_{0,\Omega}\|\nabla (\phi_\delta - i_h(\phi_\delta))\|_{0,\Omega}\Big)\\
		&\leq\,  0 + \|f\|_{0,\Omega}\|\mathfrak{z} - \phi_\delta\|_{0,\Omega}
		+ Ck^{1/2}\,\big( \|\nabla(\mathfrak{z} - \phi_\delta)\|_{0,\Omega} +\lim_{h \rightarrow 0}\|\nabla (\phi_\delta - i_h(\phi_\delta))\|_{0,\Omega}\Big)\\
		&\leq\,  \|f\|_{0,\Omega}\|\mathfrak{z} - \phi_\delta\|_{0,\Omega}+ Ck^{1/2}\, \|\nabla(\mathfrak{z} - \phi_\delta)\|_{0,\Omega}\,,
	\end{align*}
	for every $\delta>0$. Taking $\delta\to 0$ in the above inequality finishes the proof.
\end{proof}

This last result shows that the constrained part of the solution, $z_h^{(\epsilon,k)+}$, of the  bound-preserving method converges strongly to the best approximation,  in $H^1_0(\Omega)$, of the exact solution $z$ from $V^{\epsilon,k}$.  The surprising part of this is that this result \emph{is independent of whether the
values $-\epsilon$ and $k$ are appropriate or not}. This will be instrumental in the proof of convergence we
present in the next section.

\subsection{The Bound-Preserving Method for the Coupled Induction Heating Problem}

We begin by observing that, since the right hand side of \eqref{eq:strong_heat} is nonnegative, the maximum principle implies that $u \geq 0$. In addition, since we have showed, under the sole assumption that $\Omega$ is Lipschitz, that $|\nabla H|^2 \in L^2(\Omega)$; we recall that for every $\delta>0$ we have $u \in H^{3/2-\delta}(\Omega)$ and, therefore, $u \in C(\overline{\Omega})$. Thus, there exists a $K>0$ for which we have
\[
  0 \leq u(x) \leq K, \qquad \forall x\in \bar\Omega.
\]
This motivates the use of the BPM for the heat equation in the induction heating problem. We describe this approach next. 

The BPM for the induction heating problem reads as follows.  Choose $k,\epsilon>0$ with $\epsilon \ll k$. Then we need to find $u_h^{\epsilon,k}\in V_h$ and $H_h\in W_{h,\C}$ such that $\tilde{H}_h = H_h - H_\circ \in V_{h,\C}$ and, for all $(Q_h, v_h) \in V_{h,\C} \times V_h$:
\begin{subequations}\label{eq:BPM_full_form_coupled}
	\begin{align}
		(\nabla \tilde{H}_h, \nabla Q_h)_{\Omega, \C} +  (i\omega\mu(u_h^{(\epsilon,k)+}) \tilde{H}_h, Q_h)_{\Omega, \C} &= (i\omega \mu(u_h^{(\epsilon,k)+})H_\circ, Q_h)_{\Omega, \C}, \label{eq:magn_fem_BPM_coupled}\\
		(\nabla u_h^{(\epsilon,k)+}, \nabla v_h)_{\Omega} + s(u_h^{(\epsilon,k)-}, v_h) &= (|\nabla H_h|^2,v_h)_{\Omega}\,.\label{eq:heat_fem_BPM_coupled}
	\end{align}
\end{subequations}
Here, $s(\cdot,\cdot)$ is the stabilising bilinear form defined in \eqref{s-definition}. It is important to notice that $u_h^{(\epsilon,k)+}$ is truncated as defined in \eqref{k-trunc-def}, and,  that $H_h$ is implicitly dependent on $k$ and $\epsilon$ via $u_h^{(\epsilon,k)+}$.

We start by showing  existence of solutions for \eqref{eq:BPM_full_form_coupled}.
  
\begin{theorem}
	If $\Omega \subset \R^2$ is a Lipschitz polygonal domain then, for every $h>0$,  the problem \eqref{eq:BPM_full_form_coupled} has at least one solution.
\end{theorem}

\begin{proof} The proof follows very similar lines as those of Theorem~\ref{Theo:Exis-Galerkin}. So, we just detail the differences.  Once again,
we define the mapping 
\begin{equation*}
		\tilde{M}_h\coloneqq \tilde{M}_{2,h} \circ M_{1,h}:V_h^{\epsilon,k} \rightarrow V_h^{\epsilon,k}\;,\; \hat{w}_h \mapsto w_h^{(\epsilon,k)+}, \quad \text{ where } M_{1,h}:\hat{w}_h \mapsto \varphi_h, \qquad \tilde{M}_{2,h} : \varphi_h \mapsto w_h^{(\epsilon,k)+}.
	\end{equation*}
	Here, for a given $\hat{w}_h \in V_h^{\epsilon,k}$, we have set $M_{1,h}(\hat{w}_h)=\varphi_h$ and $\tilde{M}_{2,h}(\varphi_h)=w_h^{(\epsilon,k)+}$, where  $(\varphi_h,  w_h) \in W_{h, \C} \times V_h$ solves: $\tilde{\varphi}_h \coloneqq \varphi_h - H_\circ \in V_{h, \C}$, and
	\begin{align}
		(\nabla \tilde{\varphi}_h, \nabla Q_h)_{\Omega, \C}+  (i \omega \mu(\hat{w}_h)\tilde{\varphi}_h, Q_h)_{\Omega, \C} &= (i \omega \mu(\hat{w}_h)H_\circ, Q_h)_{\Omega, \C}, \label{eq:galerkin_H_decoupled-bis}\\
		(\nabla w_h^{(\epsilon,k)+}, \nabla v_h)_{\Omega}+s(w_h^{(\epsilon,k)-},v_h) &= \left(|\nabla \varphi_h|^2, v_h\right)_{\Omega},\label{eq:galerkin_u_decoupled-bis}
	\end{align}
	for all $(Q_h, v_h) \in V_{h, \C} \times V_h$.  Both discrete problems above are well-posed.
	In fact, \eqref{eq:galerkin_H_decoupled-bis} is  the exact same equation as in the proof of Theorem~\ref{Theo:Exis-Galerkin}.  
	The well-posedness of \eqref{eq:galerkin_u_decoupled-bis} is 
	proven in \cite[Theorem~3.2]{Barrenechea-2024}. 
	This shows that the mapping $\tilde{M}_h$ is well-defined.
	
	In order to prove the continuity of $\tilde{M}_h$ 
	we first use that $M_{1,h}$ is just the restriction to $V_h^{\epsilon,k}$ of the mapping
	defined in the proof of Theorem~\ref{Theo:Exis-Galerkin} and
	thus it is continuous, and the uniform bound  $\|\nabla\varphi_h\|_{0,\Omega}\le C_1$ from Lemma~\ref{lem:H_h_bounded_in_H1} still holds.  
	In addition, taking $v_h=w_h^{(\epsilon,k)+}$ as a test function in \eqref{eq:galerkin_u_decoupled-bis},
	using that $s(w_h^{(\epsilon,k)-},w_h^{(\epsilon,k)+})\ge 0$ (see \cite[Theorem~3.2]{Barrenechea-2024}),
	and the H\"older  inequality, the following bound follows 
	\begin{equation}\label{Eq:unif-bound-uhk}
	| w_h^{(\epsilon,k)+} |^2_{1,\Omega}\le | w_h^{(\epsilon,k)+} |^2_{1,\Omega} + s(w_h^{(\epsilon,k)-},w_h^{(\epsilon,k)+})= \left(|\nabla \varphi_h|^2, w_h^{(\epsilon,k)+}\right)_{\Omega}
	\le \|\nabla \varphi_h\|_{0,\Omega,\C}^2\, k\le C_1^2k\,.
	\end{equation}
	
	Finally,  to prove the continuity of $\tilde{M}_{2,h}$,  let $Q_{h,1},\ Q_{h,2}\in V_{h,\C}$, and let $x_{h,i}^{(\epsilon,k)+}=\tilde{M}_{2,h}(Q_{h,i})$, for $i=1,2$.
	Using the variational inequality characterisation of the bound-preserving method given by \eqref{VI-scalar} we rewrite \eqref{eq:galerkin_u_decoupled-bis} as
	\[
	  (\nabla x_{h,i}^{(\epsilon,k)+}, \nabla( v_h - x_{h,i}^{(\epsilon,k)+}) )_\Omega \geq ( |\nabla Q_{h,i}|^2, v_h - x_{h,i}^{(\epsilon,k)+})_\Omega, \qquad \forall v_h \in V_h^{\epsilon,k}.
	\]
	We then set $v_h = x_{h,2}^{(\epsilon,k)+}$ in the  equation with $i=1$, and vice versa, add the resulting inequalities, and obtain
	\begin{align*}
	  \|\nabla ( x_{h,1}^{(\epsilon,k)+} - x_{h,2}^{(\epsilon,k)+})\|^2_{0,\Omega} &\leq\, \| |\nabla Q_{h,1}|^2-|\nabla Q_{h,2}|^2\|_{0,1,\Omega}\|x_{h,1}^{(\epsilon,k)+} - x_{h,2}^{(\epsilon,k)+}\|_{0,\infty,\Omega}\\
    &\le\, 2k \, \| |\nabla Q_{h,1}|^2-|\nabla Q_{h,2}|^2\|_{0,1,\Omega}\,,
	\end{align*}
	and so the mapping $\tilde{M}_{2,h}$ is continuous.

Collecting all the above results,  we use the uniform bound given by \eqref{Eq:unif-bound-uhk} and define
\[
  K \coloneqq \left\{ v_h\in V_h^{\epsilon,k}: |v_h|_{1,\Omega}\le C_1\sqrt{k} \right\}.
\]
Then, $\tilde{M}_h$  maps
$K$ onto $K$, and is continuous. Hence, applying Brouwer's fixed point theorem, it follows that $\tilde{M}_h$ has a fixed point $u_h^{(\epsilon,k)+}\in V_h^{\epsilon,k}$, and defining $H_h =M_{1,h}(u_h^{(\epsilon,k)+})$ and $u_h^{\epsilon,k}:=u_h^{(\epsilon,k)+}+ u_h^{(\epsilon,k)-}$
 (where $u_h^{(\epsilon,k)-}$ is the complementary part given by \eqref{eq:heat_fem_BPM_coupled}), then $(H_h,u_h^{\epsilon,k})$ solves \eqref{eq:BPM_full_form_coupled}.
\end{proof}

We now prove the main result of this section, namely, that the solution of \eqref{eq:BPM_full_form_coupled} converges, up to a subsequence, to a solution of \eqref{eq:weak_existence_form} as $h\to 0$ and 
$k\to \infty$ (in that order), and for every $\epsilon >0$.
In order to prove it, the following higher integrability result will be essential. Its proof is postponed to Appendix~\ref{sec:RegularityObstacle}.

\begin{lemma}\label{conj:conjecture}
Let $F \in L^2(\Omega)$ be such that $F \geq 0$ almost everywhere in $\Omega$. Let $\Phi\in V^{\epsilon,k}$ be the unique solution of
\begin{equation}\label{limit-ineq-uhatk}
	(\nabla \Phi,\nabla (v-\Phi))_\Omega\ge (F, v-\Phi)_\Omega, \qquad\forall\, v\in V^{\epsilon,k}\,.
\end{equation}
	Then, $\Phi\ge 0$ almost everywhere in $\Omega$,  and, for every $p \in [P',P]$ with $P$ as in \eqref{eq:JKWm1p}, we have $\Phi \in W^{1,p}(\Omega) \cap H^1_0(\Omega)$, and there exists a constant $C(p)>0$, independent of $\epsilon$ and $k$,  such that
	\begin{equation*}
		\|\Phi\|_{1,p, \Omega} \leq C(p)\,\|F\|_{0,\Omega }\, .
	\end{equation*}
\end{lemma}

With this at hand, we are ready to prove convergence.

\begin{theorem}
Let $\{(H_h, u_h^{(\epsilon,k)+ })\}_{h>0}$ be a family of solutions of \eqref{eq:BPM_full_form_coupled}. Then, up to a subsequence,  we have that
$(H_h, u_h^{(\epsilon,k)+ })\to (H,u)$ strongly in $H^1(\Omega;\C)\times H^1_0(\Omega)$ as $h\to 0$, and $k\to \infty$ (in that order), and for every $\epsilon >0$.  In addition, $(H,u)$ solves
\eqref{eq:weak_existence_form}.
\end{theorem}
\begin{proof}
Fix $k>0$. Taking $Q_h=H_h$ as test function in \eqref{eq:magn_fem_BPM_coupled} and $v_h=u_h^{(\epsilon,k)+}$ in \eqref{eq:heat_fem_BPM_coupled},
using \eqref{Eq:uni-bound-Mh} and \eqref{Eq:unif-bound-uhk} we see that 
\begin{equation}
	\|H_h\|_{1, \Omega, \C} \leq C_H \qquad \text{ and } \qquad |u_h^{(\epsilon,k)+}|_{1, \Omega} \leq C_1\sqrt{k}\,,
\end{equation}
where $C_H>0$ depends only on $C_4, H_\circ$, and the Poincar\'e constant, for all $h>0$. 
Therefore, for a fixed $k$, there exist subsequences, still denoted by $\{H_{h}\}_{h>0}$ and $\{u_{h}^{\epsilon,k}\}_{h>0}$, such that
\begin{align}
	\tilde{H}_{h} \rightharpoonup\hat{H}^{\epsilon,k} \quad \text{ weakly in } H_0^1(\Omega, \C), &\qquad u_h^{(\epsilon,k)+} \rightharpoonup\hat{u}^{\epsilon,k} \quad \text{ weakly in } H^1_0(\Omega),\\
	\tilde{H}_{h} \rightarrow\hat{H}^{\epsilon,k} \quad \text{strongly in } L^2(\Omega, \C), &\qquad 
	u_h^{(\epsilon,k)+} \rightarrow\hat{u}^{\epsilon,k} \quad \text{strongly in } L^2(\Omega).
\end{align}

\vspace{.15cm}
\noindent\underline{Identification of the limit for \eqref{eq:magn_fem_BPM_coupled} :}
Taking the limit in each term of \eqref{eq:magn_fem_BPM_coupled} is identical to what has been done in the proof of Theorem~\ref{Theo:Conv-Galerkin}.  We see then that $(\hat{H}^{\epsilon,k},\hat{u}^{\epsilon,k})$ satisfies the equation
\begin{equation}\label{limit-H-equation}
	(\nabla \hat{H}^{\epsilon,k}, \nabla Q)_{\Omega, \C} +  (i\omega\mu(\hat{u}^{\epsilon,k}) 
	\hat{H}^{\epsilon,k}, Q)_{\Omega, \C} = (i\omega \mu(\hat{u}^{\epsilon,k})H_\circ, Q)_{\Omega, \C}\qquad\forall\, Q\in H^1_0(\Omega;\C).
\end{equation}

In addition, once again, following exactly the same steps used in the proof of  Theorem~\ref{Theo:Conv-Galerkin} we can prove that
$\tilde{H}_h\to \hat{H}^{\epsilon,k}$ strongly in $H^1_0(\Omega;\C)$.  As a consequence, we have then that $|\nabla H_h|^2$ converges to $|\nabla \hat{H}^{\epsilon,k}|^2$
strongly in $L^1(\Omega)$. 

One final fact that needs to be noticed is that, regardless of $\hat{u}^{\epsilon,k}$, since $\hat{H}^{\epsilon,k}$ solves \eqref{limit-H-equation},  then estimate \eqref{eq:JKWm1p} shows that we have the following bound
\begin{equation}
\|\hat{H}^{\epsilon,k} \|_{1,4,\Omega,\C}\le C_2\,,
\end{equation}
where $C_2$ is the constant from \eqref{Eq:Bound-H-W14} which, fundamentally,  \emph{is independent of $k$}. 

\vspace{.15cm}
\noindent\underline{Identification of the limit for \eqref{eq:heat_fem_BPM_coupled} :}
Since  $|\nabla H_h|^2$ converges to $|\nabla \hat{H}^{\epsilon,k}|^2$
strongly in $L^1(\Omega)$, and we have that $\hat{H}^{\epsilon,k}\in W^{1,4}(\Omega;\C)$,  then 
rewriting \eqref{eq:heat_fem_BPM_coupled} equivalently as a variational inequality, we can
 pass to the limit $h \to 0$ in it, and invoke Lemma~\ref{lem:lem-non-physical} to conclude that $(\hat{H}^{\epsilon,k},\hat{u}^{\epsilon,k})$ satisfies
\begin{equation}{\label{eq:BPM_inequality}}
	(\nabla \hat{u}^{\epsilon,k}, \nabla (v - \hat{u}^{\epsilon,k}))_{\Omega} \geq (|\nabla \hat{H}^{\epsilon,k}|^2, v - \hat{u}^{\epsilon,k})_{\Omega}\,,
\end{equation}
for all $v \in V^{\epsilon,k}$. In addition, in Lemma~\ref{lem:lem-non-physical}, it is also proven that $u_h^{(\epsilon,k)+}\to \hat{u}^{\epsilon,k}$ strongly in $H^1_0(\Omega)$.

\vspace{.15cm}
\noindent\underline{Recovering an equality from \eqref{eq:BPM_inequality} for any $\epsilon>0$ and $k$ large enough:} So far, we have proven that $(\hat{H}^{\epsilon,k},\hat{u}^{\epsilon,k})$
solves \eqref{limit-H-equation}, \eqref{eq:BPM_inequality}. The last step of the proof is to show that, for $k$ large enough, \eqref{eq:BPM_inequality} implies \eqref{eq:heat_fem_BPM_coupled}, thus proving that, for any $\epsilon>0$ and $k$
large enough, $(\hat{H}^{\epsilon,k},\hat{u}^{\epsilon,k})$ in fact solves \eqref{eq:weak_existence_form}.

We start by noticing that, as a consequence of Lemma~\ref{conj:conjecture}, and thanks to the Sobolev embedding theorem we have the following uniform bound for $\hat{u}^{\epsilon,k}$:
\begin{equation}\label{uniform-bound-inequality}
	\hat{u}^{\epsilon,k}\ge 0\;\textrm{a.e. in}\; \Omega\quad\textrm{and}\quad
	\|\hat{u}^{\epsilon,k}\|_{0, \infty, \Omega}\leq C\, \|\hat{u}^{\epsilon,k}\|_{1, 4, \Omega}\le C\| \nabla \hat{H}^{\epsilon,k}\|_{0, 4, \Omega, \C}^{2} \leq \tilde{C},
\end{equation}
where $\tilde{C}$ is \emph{independent of $\epsilon$ and $k$}.

Now let $\tilde{k} \geq 2\tilde{C}$, and  let us consider $\hat{u}^{\epsilon,\tilde{k}}$, which  thanks to \eqref{uniform-bound-inequality} also belongs to
$V^{0,\tilde{C}}$. Next, let $w \in \mathcal{D}(\Omega)$ and let
\begin{equation}
\tilde{v} \coloneqq \hat{u}^{\epsilon,\tilde{k}} + \delta w\quad\textrm{with}\quad	\delta = \min \left\{\frac{\tilde{C}}{2\|w\|_{0, \infty, \Omega}}, \frac{\epsilon}{2 \|w\|_{0, \infty, \Omega}}\right\}>0.
\end{equation}
We notice that since, $\hat{u}^{\epsilon,\tilde{k}} \ge 0$, we have that
\begin{equation}
	\tilde{v} = \hat{u}^{\epsilon,\tilde{k}} + \delta w \geq \delta w \geq -\frac{\epsilon}{2\|w\|_{0, \infty, \Omega}}\|w\|_{0,\infty,\Omega} \geq -\epsilon,
\end{equation}
and also using that $\hat{u}^{\epsilon,\tilde{k}}\in V^{0,\tilde{C}}$,
\begin{equation}
	\tilde{v} = \hat{u}^{\epsilon,\tilde{k}} + \delta w \leq \tilde{C}+ \delta \|w\|_{0, \infty, \Omega} \leq \tilde{C} + \frac{\tilde{C}}{2 \|w\|_{0, \infty, \Omega}} \|w\|_{0, \infty, \Omega}  \leq 2 \tilde{C}  \leq \tilde{k},\qquad\textrm{a.e. in}\;\Omega\,.
\end{equation}
Hence, $\tilde{v} \in V^{\epsilon,\tilde{k}}$ and so it is an acceptable test function to be used in \eqref{eq:BPM_inequality} to get
\begin{equation*}
    (\nabla \hat{u}^{\epsilon,\tilde{k}}, \nabla (\tilde{v} - \hat{u}^{\epsilon,\tilde{k}}))_{\Omega} \geq (|\nabla \hat{H}^{\epsilon,\tilde{k}}|^2,  \tilde{v} - \hat{u}^{\epsilon,\tilde{k}})_{\Omega} \,,
    \end{equation*}
which implies that $ \hat{u}^{\epsilon,\tilde{k}}$ satisfies
 \begin{equation}
   (\nabla \hat{u}^{\epsilon,\tilde{k}}, \nabla w)_{\Omega}  \geq  (|\nabla \hat{H}^{\epsilon,\tilde{k}}|^2, w)_{\Omega}, \label{ge-in-ineq}
\end{equation}
for every $w \in \mathcal{D}(\Omega)$.

 Using an analogous argument we can see that $\tilde{z} \coloneqq \hat{u}^{\epsilon,\tilde{k}} - \delta w \in V^{\epsilon,\tilde{k}}$, and
  using it as test function in \eqref{eq:BPM_inequality} and proceeding as before we get to
 \begin{equation}\label{le-in-ineq}
  (\nabla \hat{u}^{\epsilon,\tilde{k}}, \nabla w)_{\Omega}  \leq  (|\nabla \hat{H}^{\epsilon,\tilde{k}}|^2, w)_{\Omega}\qquad\forall\, w\in \mathcal{D}(\Omega)\,.
 \end{equation}

Summarising, from \eqref{ge-in-ineq} and \eqref{le-in-ineq} we conclude that  $\hat{u}^{\epsilon,\tilde{k}}$ satisfies
\begin{equation}
	(\nabla \hat{u}^{\epsilon,\tilde{k}}, \nabla w)_{\Omega} =  (|\nabla \hat{H}^{\epsilon,\tilde{k}}|^2, w)_{\Omega} \qquad \forall \ w \in H^1_0(\Omega),
\end{equation}
which is, in fact, \eqref{eq:heat_fem_BPM_coupled}.  As a conclusion, for any $\epsilon>0$ and $\tilde{k}\ge 2\tilde{C}$,  $(H_h^{\epsilon,\tilde{k}},u_h^{\epsilon,\tilde{k}})$ converges to
$(H^{\epsilon,\tilde{k}},\hat{u}^{\epsilon,\tilde{k}})$, which is a solution of  \eqref{eq:weak_existence_form}.
\end{proof}

%

\section{Numerical Experiments}\label{sec:numerics}

In this section we report the results of a selection of numerical experiments that showcase orders of convergence for the BPM for certain problems. We illustrate the main result found in 
Lemma \ref{lem:lem-non-physical}, and show that when truncating at a non-physical height, the BPM converges to the $H^1_0(\Omega)$ projection of the unique solution into the closed convex set $V^{\epsilon, k}$. We also show that the BPM is more accurate than the standard Galerkin method when resolving boundary layers on coarse and/or non-Delaunay meshes.  This is particularly
relevant to induction heating as  the magnetic field develops a boundary layer, known as the skin effect (see, e.g., \cite[Chap.~8]{Touzani-Rappaz-2014}), that also generates a very sharp thermal boundary layer, and anisotropic meshes are advantegeous in such situations.
We then continue by reporting the numerical results in two test cases in induction heating. The first one is an academic case
that fits exactly the setup of the problem analysed in this work, while the last case is a time-dependent case closer to the one that can be found in realistic applications.

We start by describing the implementation of the BPM for the scalar problem: find $z_h\in V_h$ such that
\begin{equation}\label{Gal-scalar}
(\nabla z_h, \nabla v_h)_\Omega = (f,v_h)_\Omega\qquad\forall\, v_h\in V_h\,,
\end{equation}
where $f\in L^2(\Omega)$ is a given datum. 
We implement the BPM by using the following Richardson iteration:  let $z_h^0$ solve \eqref{Gal-scalar}, 
let the damping parameter $\omega_r \in (0, 1]$, then iteratively solve 
\begin{equation}\label{eq:richardson}
	(\nabla z_h^{n+1}, \nabla v_h)_{\Omega} = (\nabla (z_h^n)^{(\epsilon, k)+}, \nabla v_h)_{\Omega}  +\omega_r(( f, v_h)_{\Omega} - (\nabla (z_h^n)^{(\epsilon, k)+}, \nabla v_h)_{ \Omega} - s((z_h^n)^{(\epsilon, k)-}, v_h)),
\end{equation}
for all $v_h \in V_h$. The algorithm is terminated when $\|z_h^{n+1} - z_h^n \|_{0, \Omega} \leq 10^{-12}$, and  $z_h^{(\epsilon, k)+}:= (z_h^{n+1})^{(\epsilon, k)+}$ is the resulting numerical solution. 

The BPM has been implemented using \texttt{Python} code with the software package \texttt{FEniCSx v.0.10.0} \cite{FEniCSx}. The resulting linear systems are solved using an LU solver within the \texttt{petsc} module of \texttt{FEniCSx}. In all tests we take $\alpha = 1$, $\omega_r = 0.5$, and $\epsilon = 10^{-12}$ unless otherwise stated. The full code for the BPM and these experiments can be found in the GitHub repository \cite{MacKenzie-2025}.

%

\subsection{Convergence when Truncating at a Non-Physical Bound}
Let $\Omega=(0,1)^2$ and consider the problem
\begin{equation}\label{eq:chopped_strong_form}
	\begin{aligned}
		-\Delta z &= f \quad \text{ in } \Omega, \\
		z&=0 \quad \text{ on } \partial \Omega.
	\end{aligned} 
\end{equation}
The function $z(x,y) = \sin(\pi x)\sin(\pi y)$ is the exact solution for \eqref{eq:chopped_strong_form} when $f(x,y) = 2\pi^2\sin(\pi x)\sin(\pi y)$. Additionally, $z(x,y) \in (0,1]$ for all $(x,y) \in \Omega$. 
For this experiment we consider $k = 1/2$ and show that the numerical solution $z_h^{(\epsilon, k)+}$ converges to the best approximation of $z$ (in the $H^1_0(\Omega)$-norm) in the subset 
\begin{equation*}
	V^{\epsilon, 1/2} := \left\{v \in H^1_0(\Omega) \ : \ -\epsilon\le v(x) \leq 1/2\; \textrm{a.e. in}\;\Omega \right\}.
\end{equation*}
Since this best approximation is not computable analytically, we compare the results to a numerical solution in an extremely refined  mesh containing approximately
$ 1.8 \times 10^7$ $\mathbb{P}_1$ elements. A visualisation of the numerical solution $z_h^{(\epsilon, k)+}$ for $h=0.00625$, and corresponding convergence results over a sequence of uniformly refined meshes are shown in Figure \ref{fig:BPM_cut_paraview} and \ref{fig:BPM_cut_convergence} respectively.  The results support the findings in Lemma \ref{lem:lem-non-physical}, which states that $z_h^{(\epsilon, k)+}$ is the best approximation of $z$ in the set $V_h^{\epsilon, 1/2}$, and as $h \rightarrow 0$,  $z_h^{(\epsilon,k)+}$ converges to the projection of $z$ onto $V^{\epsilon, 1/2}$. Table~\ref{tab:chopped_k} contains these results and we report the number of iterations for \eqref{eq:richardson} to converge. We observe that the algorithm presents a robust behaviour with respect to the iteration numbers.

\begin{figure}[h!]
	\centering
	\begin{subfigure}[t]{0.45\textwidth}
		\centering
		\includegraphics[width=\textwidth]{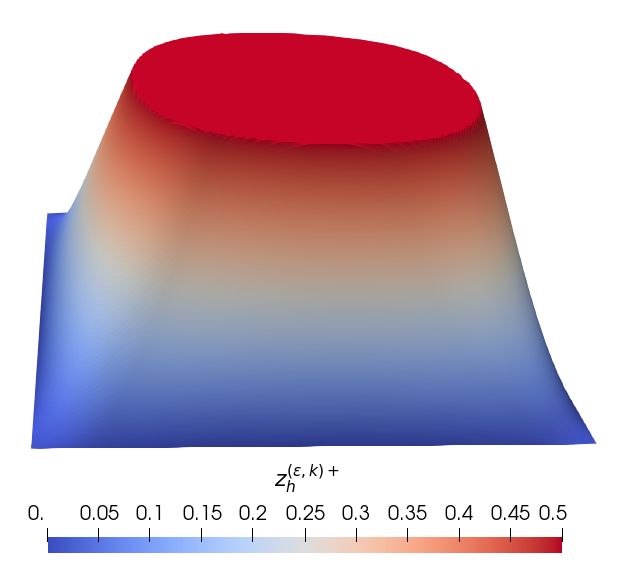}
		\caption{Plot of $z_h^{(\epsilon,k)+}$ over a uniform mesh with $h=0.00625$.}
		\label{fig:BPM_cut_paraview}
	\end{subfigure}
	\hfill
	\begin{subfigure}[t]{0.45\textwidth}
		\centering
		\includegraphics[width=\textwidth]{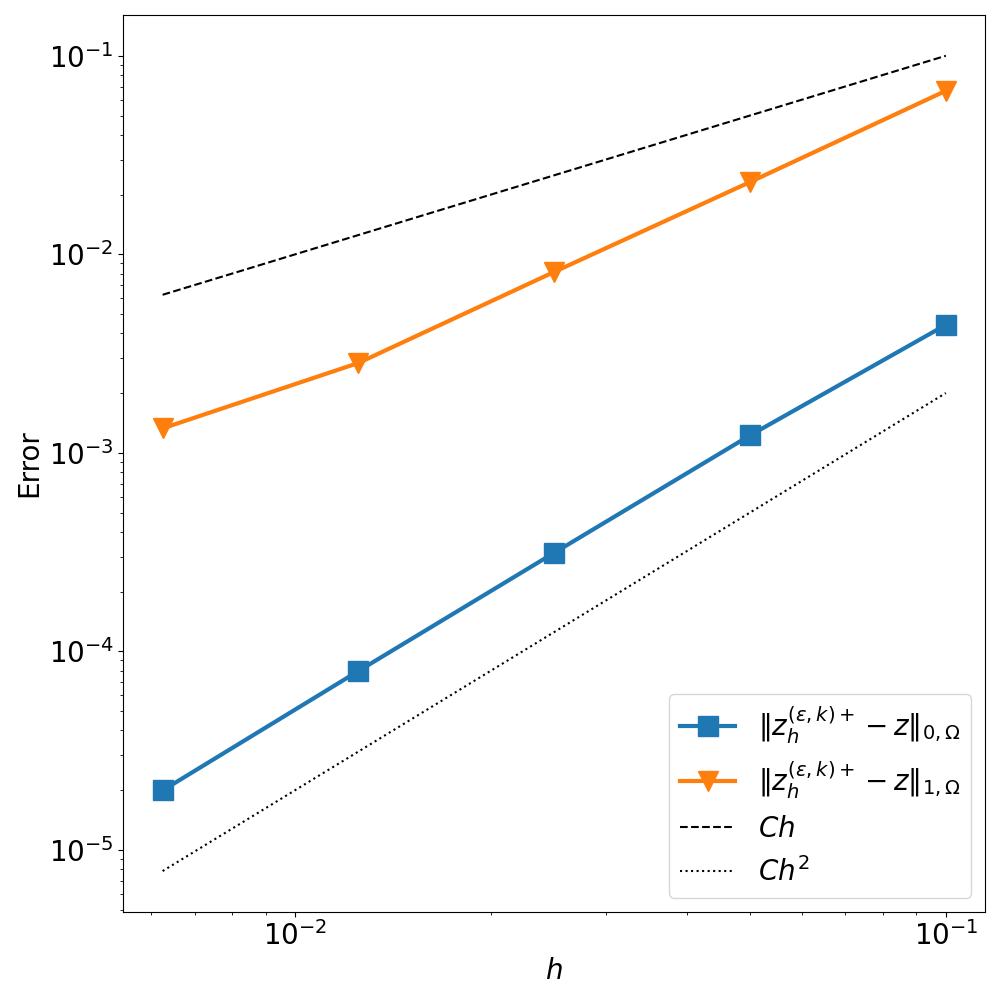}
		\caption{Convergence in the $L^2$ and $H^1$ norms. }
		\label{fig:BPM_cut_convergence}     \end{subfigure}
	\caption{The numerical solution $z_h^{(\epsilon,k)+}$ over a sequence of increasingly refined uniform meshes.}
	\label{fig:cut_graphs}
\end{figure}

\begin{table}[h!]
	\centering
\begin{tabular}{llllll}
			\toprule
			 $h$ size & N. elems & $L^2$ error & $H^1$ error & iterations \\
			\midrule
			 0.10000 & 200 & 0.004408 & 0.066598 & 7 \\
			 0.05000 & 800 & 0.001224 & 0.023110 & 7 \\
			 0.02500 & 3200 & 0.000313 & 0.008151 & 24 \\
			 0.01250 & 12800 & 0.000080 & 0.002834 & 26 \\
			 0.00625 & 51200 & 0.000020 & 0.001327  & 26 \\
			\bottomrule
	\end{tabular}
	\caption{Numerical results when solving for $z_h^{(\epsilon,k)+}$ using \eqref{eq:richardson}.}
	\label{tab:chopped_k}
\end{table}

%

\subsection{Convergence on a Non-Delaunay Mesh for Solutions with a Boundary Layer}\label{sec:theoretical_nondelaunay_test}

For this experiment we also consider a scalar equation posed over $\Omega=(0,1)^2$. We choose our exact solution to be $z_{ex}(x,y) = e^{-dx} + e^{-dy}$ where $d > 0$ (see Figure \ref{fig:Exponential paraview}). 
The function $z_{ex}$ solves the following boundary-value problem:
	\begin{align}\label{eq:nonDelaunay_problem}
		-\Delta z &= -d^2(e^{-dx} + e^{-dy}) &&\text{ in } \Omega, \\
		z&=z_{ex}& & \text{ on } \partial \Omega. \nonumber
	\end{align}
For $d$ large,  $z_{ex}$ presents a boundary layer (see Figure~\ref{fig:Exponential paraview} for the depiction of $z_{ex}$ when $d=20$).  This is a similar behaviour to that of
	the temperature distribution for high-frequency problems, where the skin effect creates a sharp boundary layer both in the magnetic field and temperature (see, e.g., \cite[Chap.~8]{Touzani-Rappaz-2014}).

\begin{figure}[h!]
	\centering
	\includegraphics[width=0.5\linewidth]{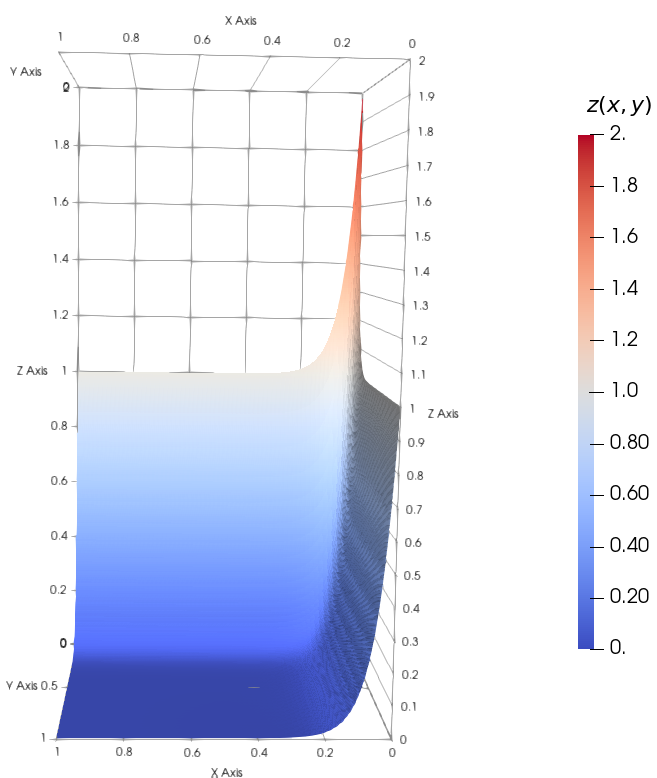}
	\caption{The function $z(x,y) = e^{-dx} + e^{-dy}$ for $d = 20$.}
	\label{fig:Exponential paraview}
\end{figure}

We approximate this problem using the BPM with the Richardson iteration \eqref{eq:richardson}.  Since,  beyond shape regularity, the BPM has no restrictions on the mesh, we use
a sequence of non-Delaunay meshes that are designed to have severely obtuse angles.  It is a well-known fact that in such case the discrete maximum principle is
not guaranteed, and, in fact it is possible to build examples where it fails, even for a Poisson problem (see, e.g., \cite{Barrenechea-John-Knobloch-2025}). The meshes depicted in Figure~\ref{fig:nondelaunay_sequence} were generated using \texttt{GMSH} \cite{GMSH} and then manually edited. 

\begin{figure}[h!]
	\centering
	\begin{subfigure}[b]{0.3\textwidth}
		\centering
		\includegraphics[width=\textwidth]{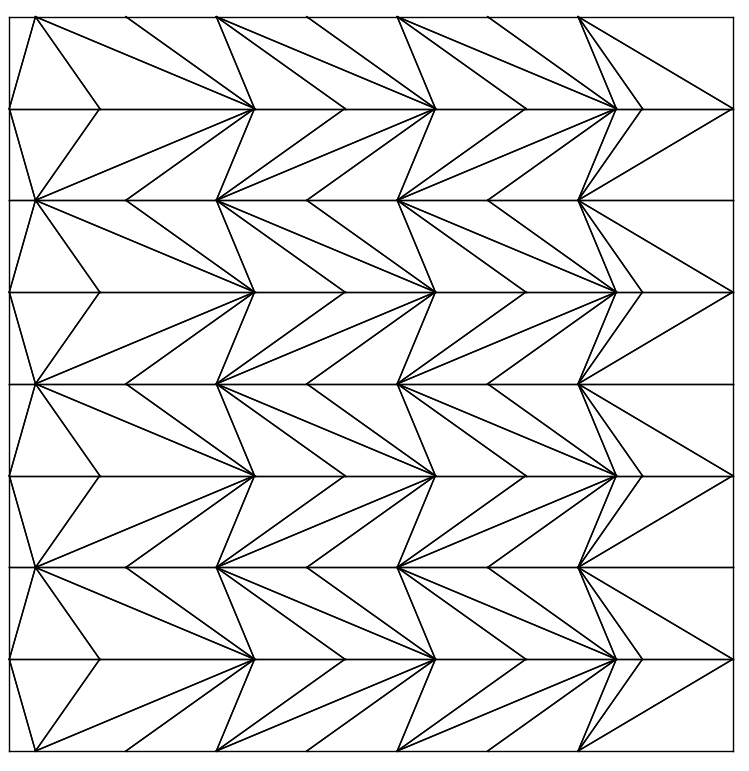}
		\caption{Iteration 1: mesh with $n=9$ nodes on the boundary.}
		\label{fig:nonDelaunay_9}
	\end{subfigure}
	\hfill
	\begin{subfigure}[b]{0.3\textwidth}
		\centering
		\includegraphics[width=\textwidth]{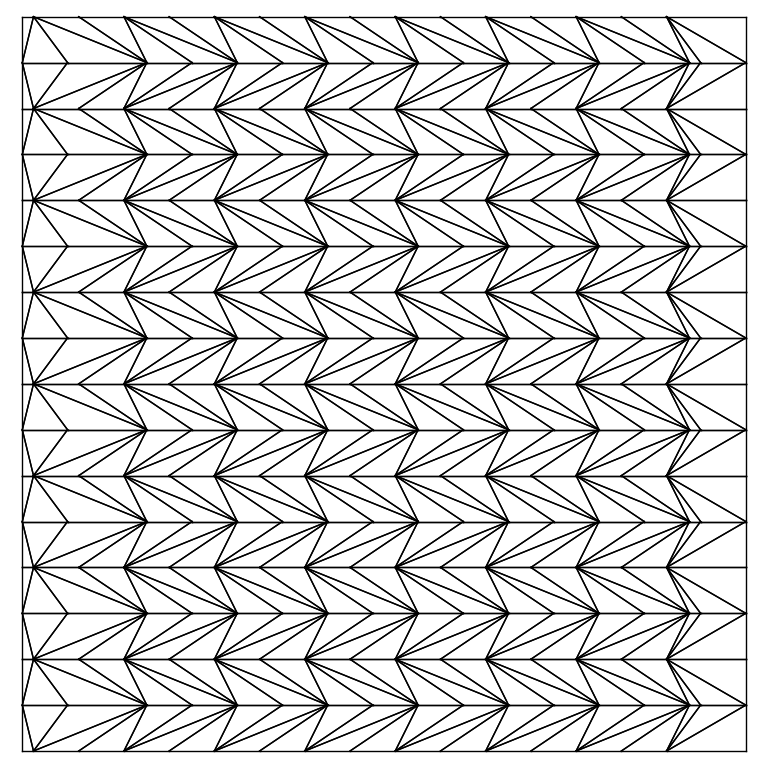}
		\caption{Iteration 2: mesh with $n=17$ nodes on the boundary.}
		\label{fig:nonDelaunay_17}
	\end{subfigure}
	\hfill
	\begin{subfigure}[b]{0.3\textwidth}
		\centering
		\includegraphics[width=\textwidth]{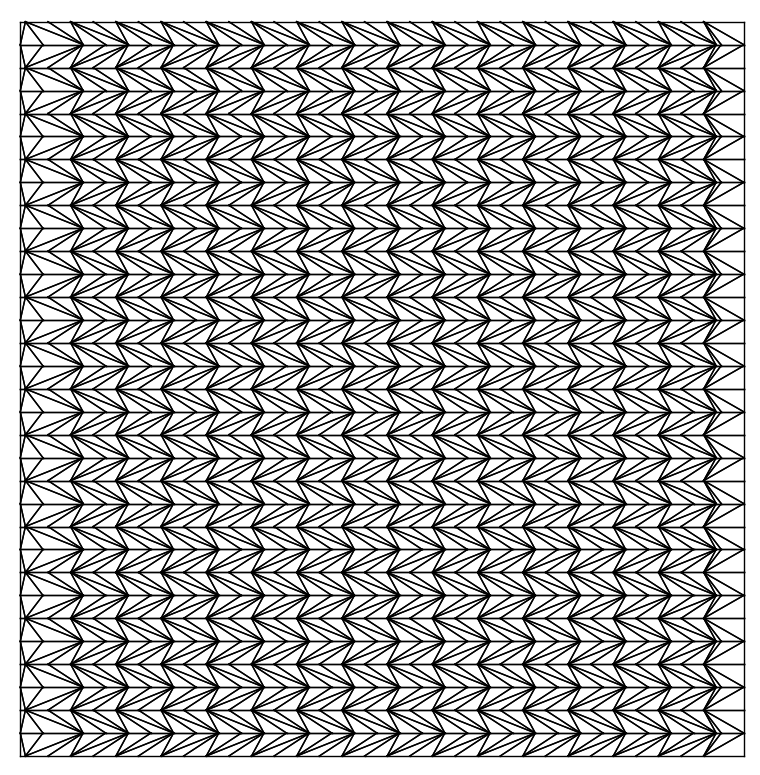}
		\caption{Iteration 3: mesh with $n=33$ nodes on the boundary.}
		\label{fig:nonDelaunay_33}
	\end{subfigure}
	\caption{A sequence of increasingly refined non-Delaunay meshes. Here, $n$ denotes the number of nodes on each boundary.}
	\label{fig:nondelaunay_sequence}
\end{figure}

We compare convergence rates of the BPM and the standard Galerkin solution given by solving \eqref{Gal-scalar} over the sequence of meshes depicted in Figure~\ref{fig:nondelaunay_sequence}. 
For the BPM we use $k=2$ and denote the solution by $z_h^{(\epsilon, k)+}$. The results are depicted in Figure \ref{fig:nonDelaunay_sequence} and show the convergence rates for different values of $d$.

From the depicted results, we can see that $z_h^{(\epsilon, k)+}$ is always more accurate than the standard Galerkin solution in all cases.  The advantage of the BPM is most clearly seen when the mesh is coarse, or the boundary layer is more prominent. In the most extreme case, with $d = 40$ on the coarsest grid, $z_h^{(\epsilon, k)+}$ is significantly more accurate than the Galerkin solution, which in the figures is denoted by $z_h^{GAL}$. The reason for this is illustrated in Figure \ref{fig:boundary_layer_galerkin_and_BPM_comparison}. Since the mesh does not guarantee that the discrete maximum principle holds for the Galerkin solution,  $z_h^{GAL}$ presents significant  undershoots. These undershoots are most significant on the coarsest mesh, where the numerical solution obtains a minimum of -3.3, and a maximum of 2, while the exact solution is  strictly positive. This illustrates a significant advantage of the BPM over the Galerkin solution.

\begin{figure}[h!]
	\centering
	\begin{subfigure}[b]{0.3\textwidth}
		\centering
		\includegraphics[width=\textwidth]{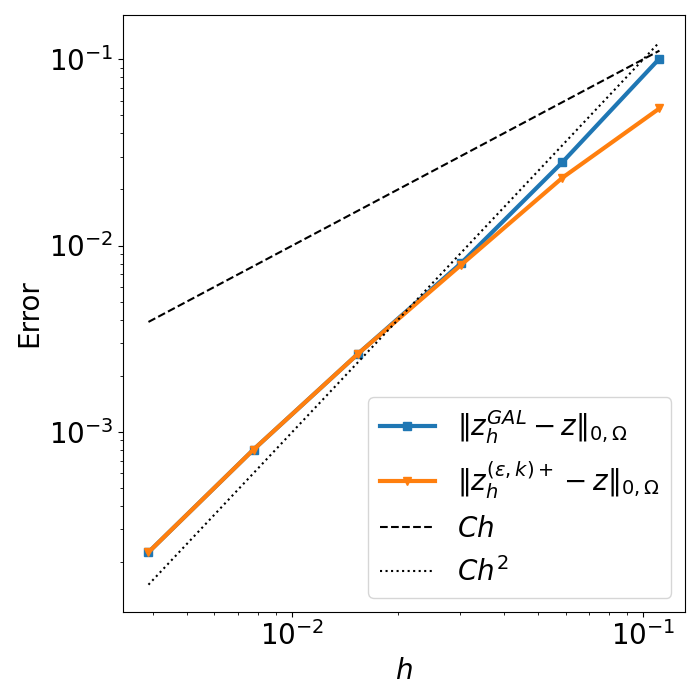}
		\caption{$d=10$.}
		\label{fig:nonDelaunay_exp_10}
	\end{subfigure}
	\hfill
	\begin{subfigure}[b]{0.3\textwidth}
		\centering
		\includegraphics[width=\textwidth]{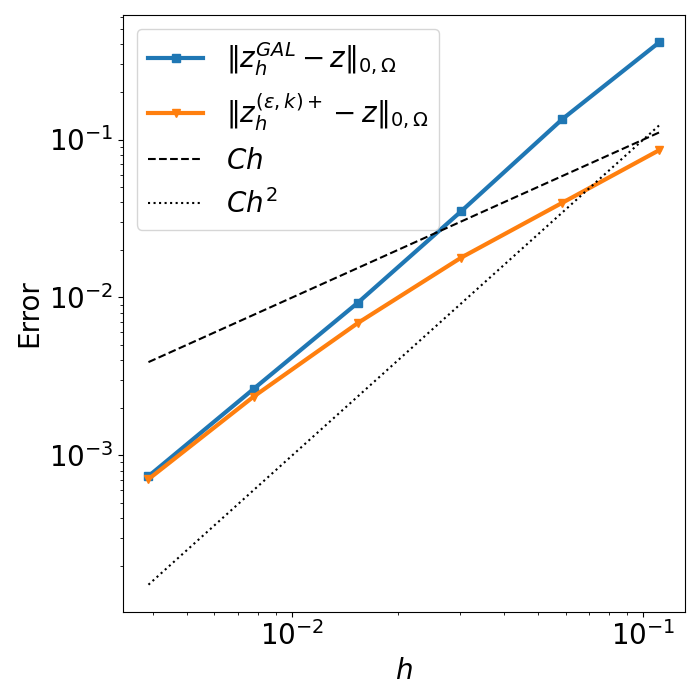}
		\caption{$d=20$.}
		\label{fig:nonDelaunay_exp_20}
	\end{subfigure}
	\hfill
	\begin{subfigure}[b]{0.3\textwidth}
		\centering
		\includegraphics[width=\textwidth]{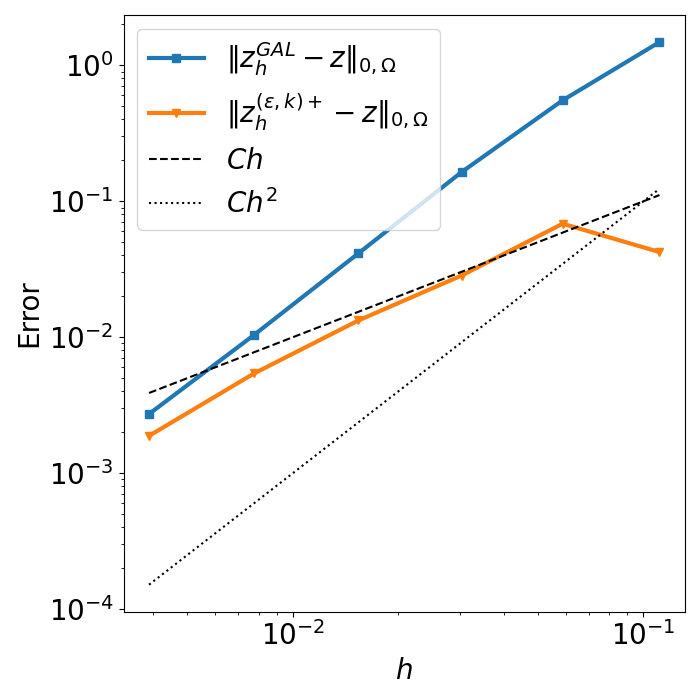}
		\caption{$d=40$.}
		\label{fig:nonDelaunay_exp_40}
	\end{subfigure}
	\caption{Comparison between $L^2(\Omega)$ errors of $z_h^{(\epsilon,k)+}$ and $z_h^{GAL}$ for different values of $d$ over a sequence of increasingly refined non-Delaunay meshes. }
	\label{fig:nonDelaunay_sequence}
\end{figure}

\begin{figure}[h!]
	\centering
	\begin{subfigure}[b]{0.3\textwidth}
		\centering
		\includegraphics[height=5cm]{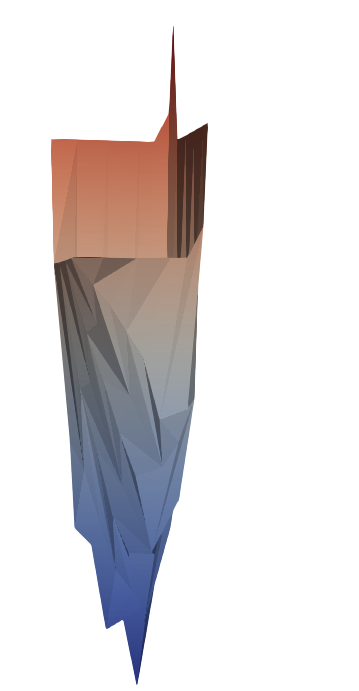}
		\caption{$z_h^{GAL}$ with $n = 9$.}
	\end{subfigure}
	\hfill
	\begin{subfigure}[b]{0.3\textwidth}
		\centering
		\includegraphics[height=5cm]{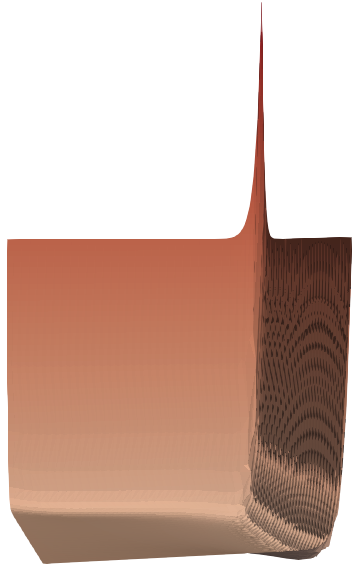}
		\caption{$z_h^{GAL}$ with $n = 65$.}
	\end{subfigure}
	\hfill
	\begin{subfigure}[b]{0.3\textwidth}
		\centering
		\includegraphics[height=5cm]{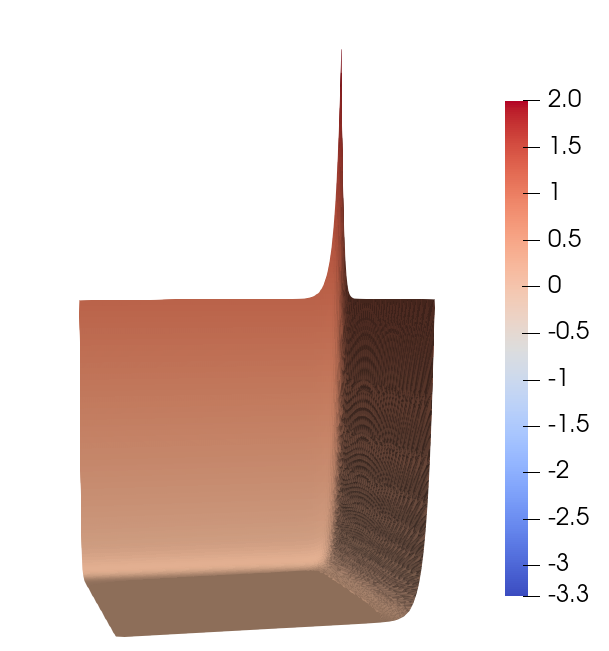}
		\caption{$z_h^{GAL}$ with $n = 257$.}
	\end{subfigure}
	
	\vspace{0.5cm}
	
	\begin{subfigure}{0.3\textwidth}
		\centering
		\includegraphics[height=5cm]{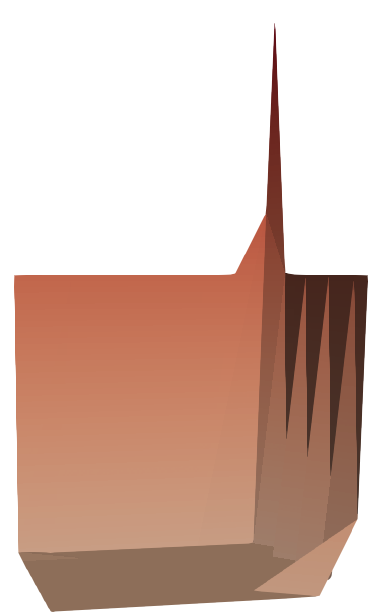}
		\caption{$z_h^{(\epsilon, k)+}$ with $n = 9$.}
	\end{subfigure}\hfill 
	\begin{subfigure}{0.3\textwidth}
		\centering
		\includegraphics[height=5cm]{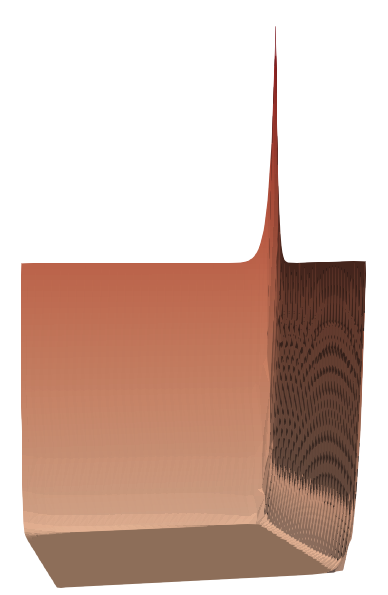}
		\caption{$z_h^{(\epsilon, k)+}$ with $n = 65$.}
	\end{subfigure}\hfill 
	\begin{subfigure}{0.3\textwidth}
		\centering
		\includegraphics[height=5cm]{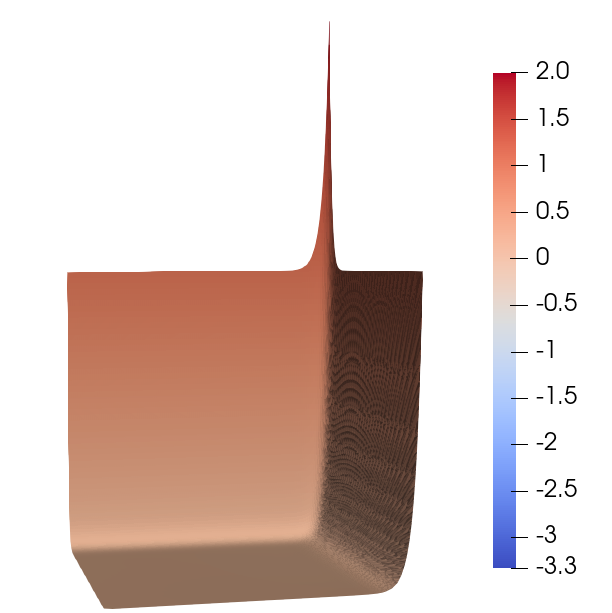}
		\caption{$z_h^{(\epsilon, k)+}$ with $n = 257$.}
	\end{subfigure}
	\caption{Visualisation of $z_h^{GAL}$ and $z_h^{(\epsilon, k)+}$ for different non-Delaunay meshes for $d=40$.}
	\label{fig:boundary_layer_galerkin_and_BPM_comparison}
\end{figure}

\subsection{Convergence of the Non-Linear Solenoidal Induction Heating Problem with a Layer}
Next, we test for convergence for the  solenoidal induction heating problem  \eqref{eq:strong_solenoidal_theory} with $\omega = 1$. We consider $\Omega=(0,1)^2$  and we choose the exact solution $(H_{ex}, u_{ex})$ and the non-linear, strictly positive function $\mu(\cdot)$ to be
\begin{itemize}
	\item $u_{ex}(x,y) = e^{-dx} + e^{-dy},$
	\item $H_{ex}(x,y) = 2 + \sin(\pi x)\sin(\pi y)$, 
	\item $\mu(s) = s^2 +1$.
\end{itemize}
Thus, $u_{ex}(x,y) \in (0,2]$ for all $(x,y) \in \Omega$, so we set $k=2$ in the BPM. This solution has a boundary layer when $d$ is large. Then, we create a manufactured solution by defining the source terms to be 
\begin{itemize}
	\item $f_H = - \Delta H_{ex} + i \omega \mu(u_{ex})H_{ex}$,
	\item $f_u = -\Delta u_{ex} - |\nabla H_{ex}|^2 $.
\end{itemize}

The nonlinear system associated with the discrete problem \eqref{eq:BPM_full_form_coupled} is solved using the fixed-point iteration given by Algorithm~\ref{alg:fixed_point}. In this experiment, we set the coupled damping parameters $\rho = 0.5$ and the Richardson damping parameter $\omega_r = 0.8$. The outer iteration is a standard fixed-point iteration, while the inner  iteration, associated with the solution of the discrete heat equation \eqref{eq:heat_fem_BPM_coupled} is performed via the Richardson method  \eqref{eq:richardson}. For the case where the plain Galerkin method \eqref{eq:discrete_full_weak_form} is used, then the 
Step~2 of Algorithm~\ref{alg:fixed_point}  is replaced by
\begin{equation*}
	(\nabla u_h, \nabla v_h)_\Omega = (|\nabla H_h^m|^2 + f_u, v_h)_\Omega \qquad \forall \ v_h \ V_h,
\end{equation*}
and Step~5 with
\begin{equation*}
	(\nabla \tilde{u}_h^m, \nabla v_h)_\Omega = (|\nabla H_h^{m+1}|^2 + f_u, v_h)_\Omega \qquad \forall \  v_h \in V_h.
\end{equation*}

\begin{algorithm}
	\caption{A fixed-point algorithm to solve the coupled solenoidal induction heating problem.}
	\label{alg:fixed_point}
	\begin{algorithmic}
		\REQUIRE {$\delta = 10^{-6}$, $M = 100$, $\rho \in (0, 1]$. }
		\STATE{Step 1: Find the initial guess $H^0_h\in V_{h,\C}$ by solving the linear system: Find $H_h^0\in W_{h, \C}$ such that $H_h^0=2$ on
		$\partial\Omega$, and :
			\begin{equation*}
				(\nabla H_h^0, \nabla Q_h)_{\Omega, \C} +  (i\omega H_h^0, Q_h)_{\Omega, \C} = (f_H, Q_h)_{\Omega, \C}, \qquad \forall \ Q_h \in V_{h,\C}.
			\end{equation*}Set $H_h^m = H_h^0$.}
		\STATE{Step 2: Find the initial guess $u^0_h \in V_h$ using the Richardson iteration to solve:
			\begin{equation*}
				(\nabla (u_h^0)^{(\epsilon,k)+}, \nabla v_h)_{\Omega} + s((u_h^0)^{(\epsilon,k)-}, v_h) = (|\nabla H_h^m|^2 + f_u ,v_h)_{\Omega} \qquad \forall\  v_h \in V_h.
			\end{equation*}
			Set $u_h^m= u_h^0$. }
		\WHILE {($\|u_h^m - u_h^{m+1}\|_{0, \Omega} > \delta$ or $\|H_h^m - H_h^{m+1}\|_{0, \Omega ,\C} > \delta$) and $m < M$}
		\STATE{Step 3: Find $H_h^m\in W_{h, \C}$ such that $H_h^m=2$ on 
		$\partial\Omega$, and :
			\begin{equation*}
				(\nabla H_h^m, \nabla Q_h)_{\Omega, \C} +  (i\omega \mu(u^{m}_h) H_h^m, Q_h)_{\Omega, \C} = (f_H, Q_h)_{\Omega, \C}, \qquad \forall \ Q_h \in V_{h,\C}.
			\end{equation*} 
		}
		\STATE{Step 4: Set $H_h^{m+1} = H_h^m + \rho (\tilde{H}_h^m - H_h^m)$.}
		\STATE{Step 5: Find $\tilde{u}_h^m\in V_h$ by using the Richardson iteration to solve:
			\begin{equation*}
				(\nabla (\tilde{u}^m_h)^{(\epsilon,k)+}, \nabla v_h)_{\Omega} + s((\tilde{u}^m_h)^{(\epsilon,k)-}, v_h) = (|\nabla H_h^{m+1}|^2 + f_u ,v_h)_{\Omega} \qquad \forall \ v_h \in V_h.
		\end{equation*} }
		\STATE{Step 6: Set $u_h^{m+1} = u_h^m + \rho \left(\tilde{u}_h^m - u_h^m\right)$. }
		\STATE{Step 7: Check the tolerance criteria, then set $u_h^m = u_h^{m+1}$ and $H_h^m = H_h^{m+1}$.} 
		\ENDWHILE
		\RETURN{($H_h^{m+1}$, $u_h^{m+1}$).}
	\end{algorithmic}
\end{algorithm}

The error curves for the Bound-Preserving Method~\eqref{eq:BPM_full_form_coupled} and the plain Galerkin scheme~\eqref{eq:discrete_full_weak_form} on the sequence
of non-Delaunay meshes illustrated in Figure \ref{fig:nondelaunay_sequence} are depicted in Figure~\ref{fig:coupled_experiment}. In there, we can observe the advantages provided
by the BPM, especially in coarser meshes.

In Table~\ref{tab:coupled_error} we report the number of iterations needed for both the plain
Galerkin and the BPM methods. We can observe these numbers are fairly mesh-independent. 

\begin{figure}[h!]
	\centering
	\includegraphics[width=\textwidth]{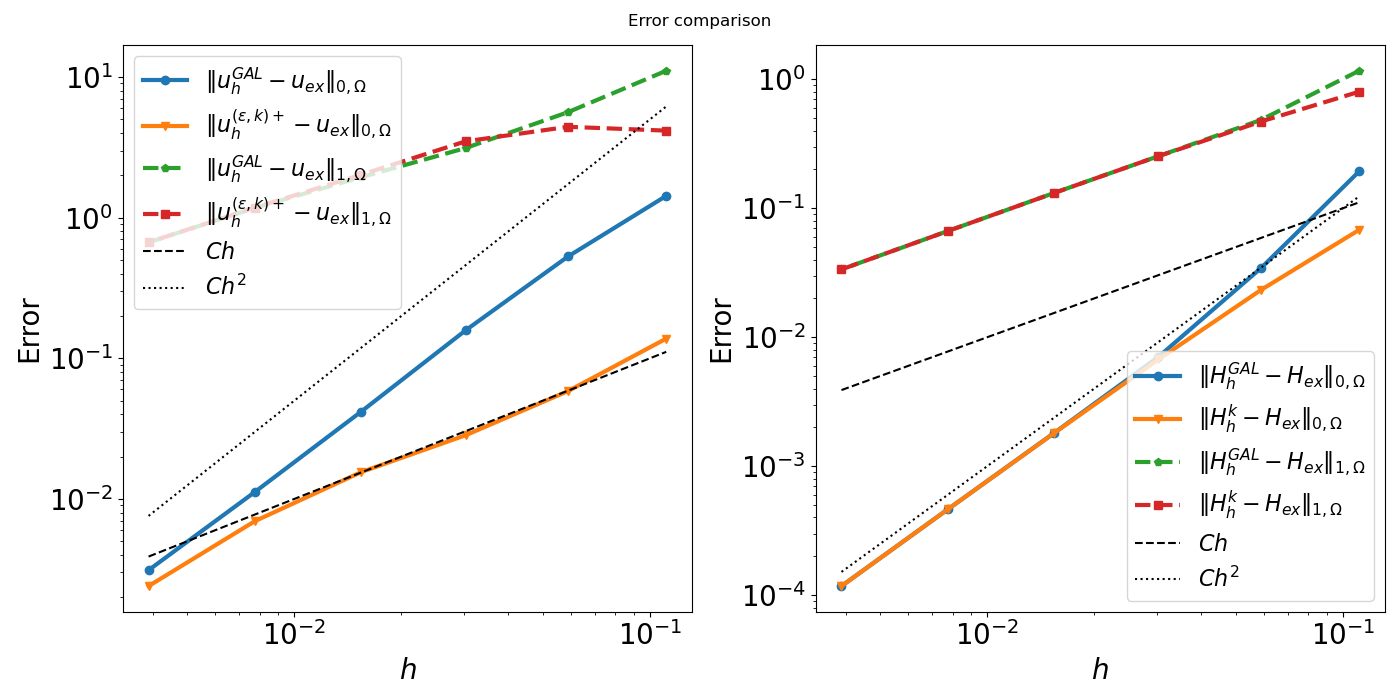}
	\caption{Convergence plot comparing the standard Galerkin method t the BPM. Both methods find the output of Algorithm \ref{alg:fixed_point} over a sequence of increasingly refined non-Delaunay meshes. The plot contains the $L^2(\Omega)$ and $H^1(\Omega)$ errors.}
	\label{fig:coupled_experiment}
\end{figure}

\begin{table}[h!]
\begin{tabular}{cccc}
	\toprule
	h size & N. elems & Gal iterations & BPM iterations \\
	\midrule
	0.111111 & 162 & 43 & 35  (9) \\
	0.058824 & 578 & 38& 33  (8) \\
	0.030303 & 2178 & 36 & 31  (8) \\
	0.015385 & 8450 & 34 & 30 (9) \\
	0.007752 & 33282 & 32 & 29  (10) \\
	0.003891 & 132098 & 30 & 28  (9) \\
	\bottomrule
\end{tabular}
\caption{Number of outer iterations for Galerkin and the BPM with $d = 40$. The number of inner Richardson iterations for the BPM is included in brackets.}
\label{tab:coupled_error}
\end{table}

\subsection{Non-Convex Domain with Non-Delaunay Anisotropic Mesh}\label{sec:non-convex_domain}
In this final section, we consider a time-dependent solenoidal model with realistic parameters. This experiment shows that the BPM works for non-Delaunay meshes on non-convex domains when applied to a real-world setting. See Figure \ref{fig:nonconvex_setup} for measurements of the domain and the corresponding mesh. We note that even though \eqref{eq:strong_solenoidal_theory} is a Dirichlet problem, the results in this section suggest that the convergence proofs in Section \ref{sec:BPM_theory} might extend to the temporal problem with radiation boundary conditions with realistic coefficients. We also show that the BPM can provide better results than the standard Galerkin method in certain scenarios when the standard Galerkin method produces spurious oscillations in the solution.
The mesh used  is an extreme case of those that would typically be used in industrial settings (as the elements'
aspect ratio is very large),  and we use it here to demonstrate the advantages of the BPM over the standard Galerkin method, and its robustness. 

\begin{figure}[h!]
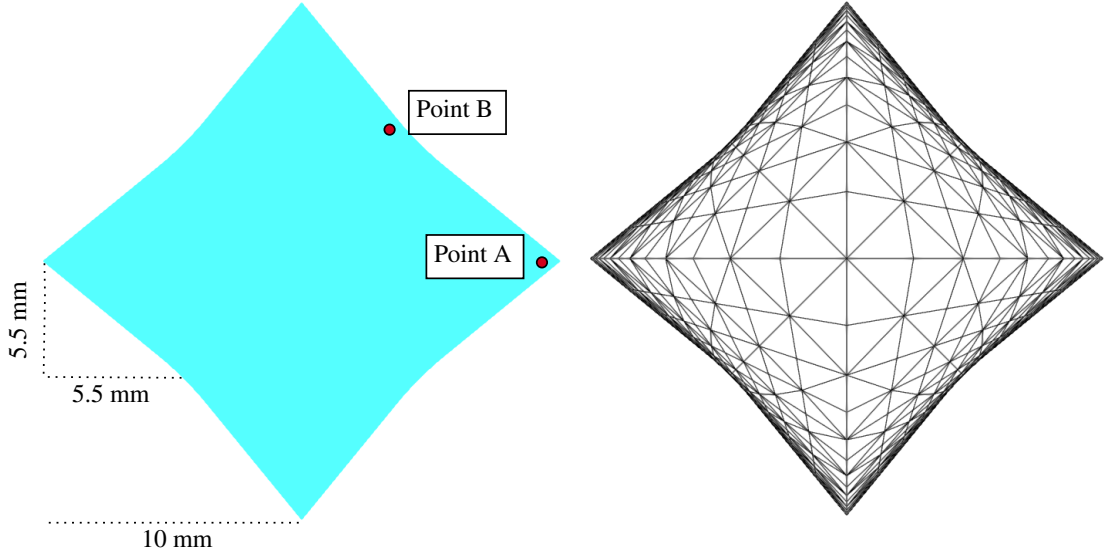

	\centering      
	\includestandalone[width=0.9\textwidth]{tikz_images/non-convex_setup}
	\caption{Illustration of non-convex domain and corresponding non-Delaunay mesh with 4761 nodes and 9248 elements.  We track the temperature at Point A, positioned at (0.0090, 0.000), and Point B, positioned at (0.00342, 0.00568), where (0,0) is at the centre of the billet, and measurements are in $\qty{}{\meter}$.}
	\label{fig:nonconvex_setup}
\end{figure}

We consider the full time-dependent induction heating equations with radiation boundary conditions, similar to the ones described in \cite{Chaboudez-1994}: 

\begin{subequations}\label{eq:full_numeric_solenoidal}
	\begin{align}
		- \text{div}(\sigma^{-1}(u) \nabla H) + i \omega \mu(u) H &= 0 && \text{ in } \Omega \times (0, T], \label{eq:strong_magn_numeric}\\
		H &= H_0 &&\text{ on } \partial \Omega\times (0, T], \label{eq:strong_magn_bc_numeric} \\
		\rho(u) C_p(u) \frac{\partial u}{\partial t} -\text{div} (\kappa(u)\nabla(u)) &= \frac{1}{2\sigma(u)}|\nabla H|^2 && \text{ in } \Omega\times (0, T], \label{eq:strong_heat_numeric} \\
		\kappa(u)\frac{\partial u}{\partial n} + \alpha(|u|^3u - u_{\text{amb}}^4) + \beta(u - u_{\text{amb}}) &=0 && \text{ on } \partial \Omega\times (0, T], \\ \label{eq:strong_heat_bc_numeric}
		u(\cdot , 0)&= u_0 && \text{ in } \Omega.
	\end{align}
\end{subequations}

Here, $\sigma(\cdot)$, $\mu(\cdot)$, $\kappa(\cdot)$, $\rho(\cdot),$ and $C_p(\cdot)$ are strictly positive, bounded, and Lipschitz functions representing the electrical conductivity, magnetic permeability, thermal conductivity, material density, and specific heat capacity, respectively. The positive constants $\omega$, $u_{\text{amb}}$, $\alpha$, and $\beta$ represent the angular frequency of the current, the ambient temperature, the radiation coefficient, and the convective coefficient, respectively.  For this experiment, we use material properties for a steel that is typically used in induction heating, using data provided by the Advanced Forming Research Centre (AFRC) at the University of Strathclyde.  See Figure \ref{fig:material_properties} for the material properties where the vertical scales have been omitted for anonymity purposes. Note that near \qty{650}{\degreeCelsius}, there is a sharp change in $\mu(\cdot)$ as the billet reaches the temperature (called the Curie point) where it becomes completely demagnetised. This can sometimes cause issues with convergence with numerical schemes, and has been previously resolved by using a `smoothed' version of \cite{Chaboudez-1994}, using a predictor corrector scheme with very small time-steps \cite{Masse-Morel-Breville-1985}, or simply not extending beyond the Curie point \cite{Bay-Labbe-2003}. 
\begin{figure}[h!]
	\centering
	\includegraphics[width=\textwidth]{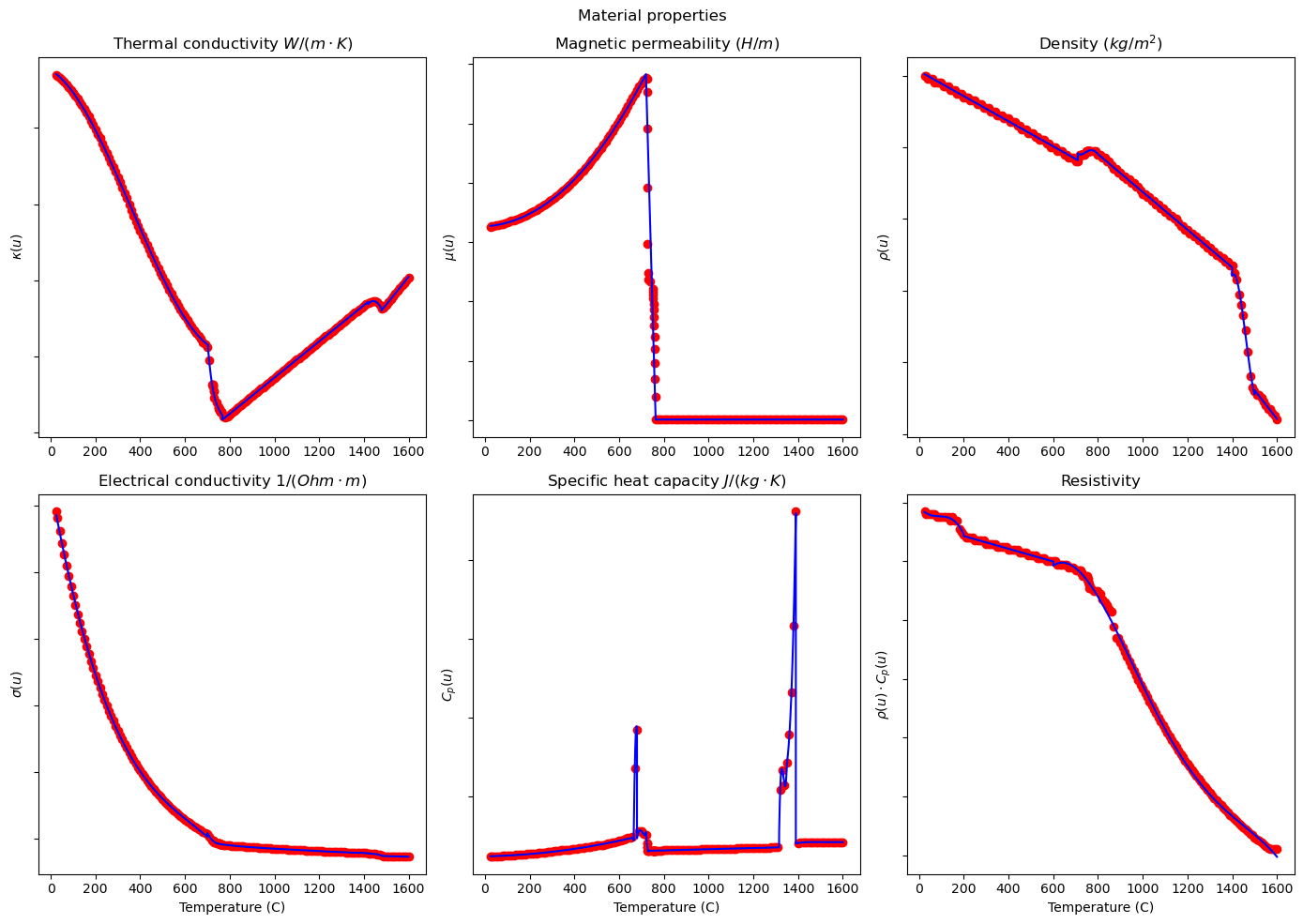}
	\caption{Material properties of a typical steel. The red dots represent the datapoints, and the blue line represents a piecewise polynomial approximation of the data.  Near \qty{650}{\degreeCelsius}, there is a sharp change in $\mu(\cdot)$ as the billet reaches the temperature (called the Curie point) where it becomes completely demagnetised.}
	\label{fig:material_properties}
\end{figure}

We set the current intensity to be \qty{417}{\ampere}, a typical intensity used for heating small workpieces, and heat the billet for 8 seconds. See Table \ref{tab:table_for_nonconvex_geometry} for all the properties used in the setup. We configure the algorithm to be as unstable as possible: we elect to use a large time-step, set a small number of maximum iterations in all fixed-point loops, and set the coupled and temperature damping parameters to be close to 1.

\begin{table}
	\centering
	\small
	\begin{tabular}{|c|c|}
		\hline  
		\textbf{Parameter} & \textbf{Value} \\
		\hline 
		Time-step $\Delta t$ & \qty{1}{\second} \\
		\hline 
		Final time $T$  &\qty{15}{\second} \\
		\hline 
		Time current active & \qty{8}{\second} \\
		\hline 
		Current $I$ &  \qty{471}{\ampere} \\
		\hline 
		Ambient temperature $u_{\text{amb}}$& $\qty{300.15}{\kelvin}$ ($\qty{25}{\degreeCelsius}$) \\
		\hline
		Convective coefficient $\beta$ & 10 \\
		\hline 
		Radiation coefficient $\alpha$ &\num{4.54e-8} \\
		\hline
	\end{tabular}
	\caption{Configuration parameters for the experiment.}\label{tab:table_for_nonconvex_geometry}
\end{table} 

\begin{figure}[h!]
	\centering
	\includegraphics[width=0.7\textwidth]{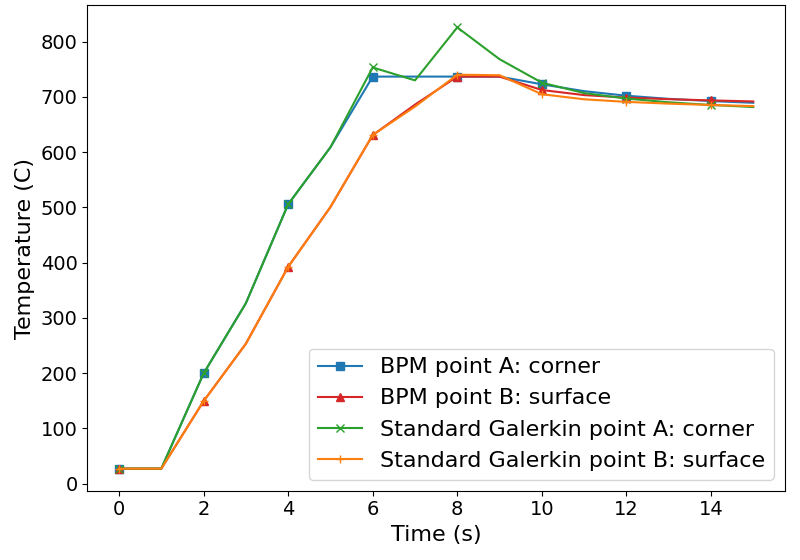}
	\caption{Results comparing the standard Galerkin method and the BPM on a non-convex domain. Here we set $k = \qty{736.85}{\celsius}$ in the BPM. }
	\label{fig:nonconvex_results}
\end{figure}

We choose control points A and B as depicted in Figure~\ref{fig:nonconvex_setup}.  The temperature at these points at each time-step are plotted in Figure \ref{fig:nonconvex_results} for both the standard Galerkin method and for the BPM. We can see that the standard Galerkin method exhibits some instability near the Curie point,  and we get some spurious oscillations at the corner of the billet. However, by choosing $k= \qty{1010}{\kelvin}$ $(\qty{736.85}{\celsius})$ in the BPM we can avoid these spurious oscillations, and obtain a much smoother and realistic solution. For the parts of the simulation where there are no numerical oscillations, the BPM and the standard Galerkin solution produce the same results. The choice of $k$ was found by running the standard Galerkin method with a more refined time-step, finding the maximum value of temperature in the simulation, and setting $k$ equal to that maximum.

\begin{figure}[h!]
	\centering
	\begin{subfigure}[t]{0.37\textwidth}
		\centering
		\includegraphics[height=5cm]{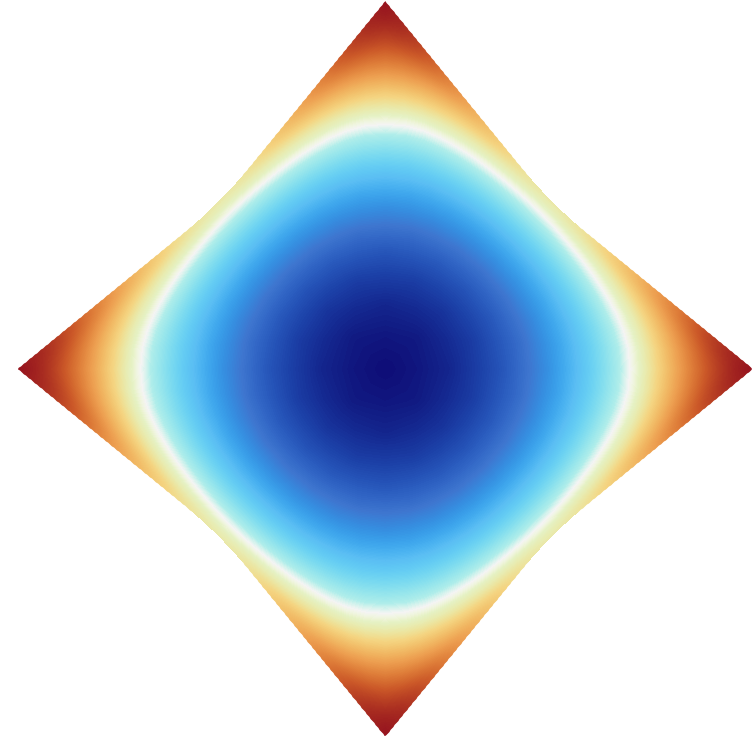}
		\caption{Galerkin solution $t = 8$.}
	\end{subfigure}
	\hfill
	\begin{subfigure}[t]{0.37\textwidth}
		\centering
		\includegraphics[height=5cm]{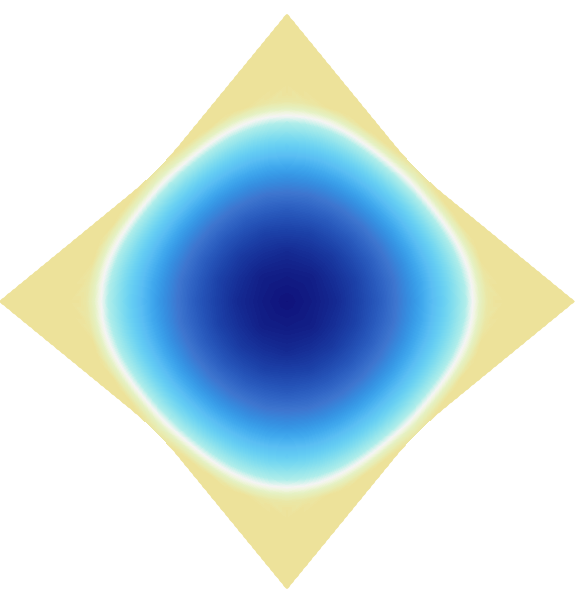}
		\caption{BPM solution at $t=8$.}
	\end{subfigure}
	\hfill
	\begin{subfigure}[t]{0.2\textwidth}
		\centering
		\includegraphics[height=5cm]{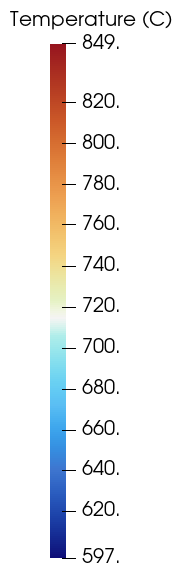}
		\caption{Temperature scale.}
	\end{subfigure}
	\caption{Comparison between the standard Galerkin method and the BPM at $t=8$.}
	\label{fig:nonconvex_paraview_t=8}
\end{figure}

\begin{figure}[h!]
	\centering
	\begin{subfigure}[t]{0.37\textwidth}
		\centering
		\includegraphics[height=5cm]{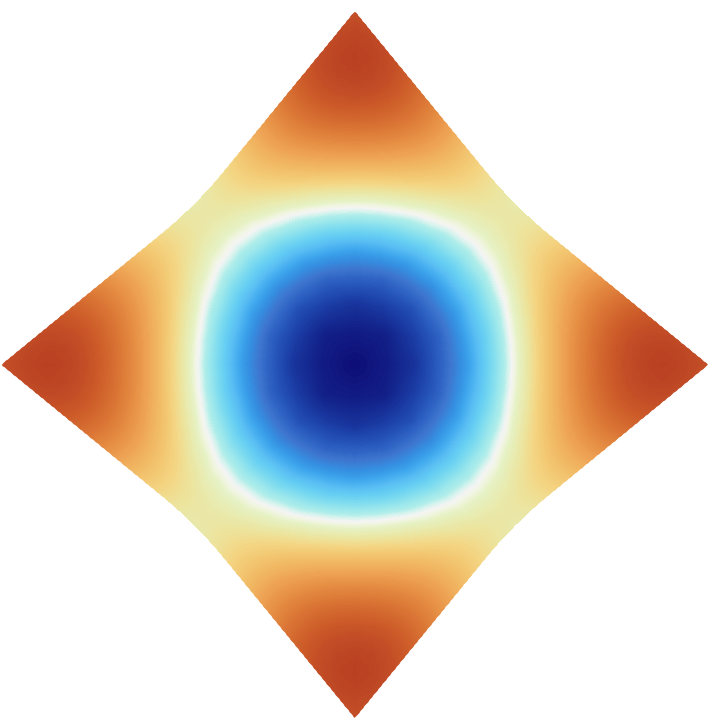}
		\caption{Galerkin solution $t = 11$.}
	\end{subfigure}
	\hfill
	\begin{subfigure}[t]{0.37\textwidth}
		\centering
		\includegraphics[height=5cm]{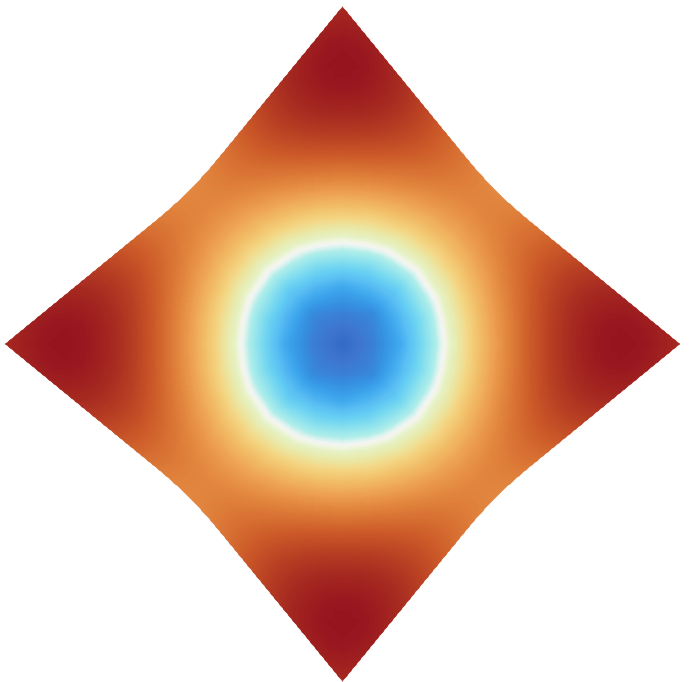}
		\caption{BPM solution at $t=11$.}
	\end{subfigure}
	\hfill
	\begin{subfigure}[t]{0.2\textwidth}
		\centering
		\includegraphics[height=5cm]{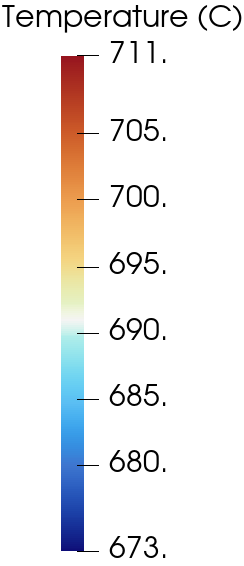}
		\caption{Temperature scale.}
	\end{subfigure}
	\caption{Comparison between the standard Galerkin method and the BPM at $t=11$.}
	\label{fig:nonconvex_paraview_t=11}
\end{figure}

Figure \ref{fig:nonconvex_paraview_t=8} illustrates the difference between the temperature distributions of the standard Galerkin solution and the BPM at the point of highest instability. We can see that the standard Galerkin solution has much higher temperatures at the corners, whereas the BPM is constrained. However, looking at Figure \ref{fig:nonconvex_paraview_t=11}, a few seconds later, after the current has been switched off, we can see that the BPM produces a temperature distribution that is higher. This suggests that the standard Galerkin solution with higher temperature at $t=8$ was indeed a numerical artefact, and is causing undershoots in the temperature distribution at later times.

We finish by commenting that although it is clear from this experiment that unwanted numerical oscillations can be suppressed by the BPM, the latter depends on the
 choice of the parameter $k$, the point where the solution is truncated. Too low a choice of $k$, and the entire solution would be suppressed and this would induce a non-physical solution. Therefore, in order to use the BPM effectively, we need to know how to make an educated choice for $k$. Nevertheless, this experiment is a proof of concept that the BPM could be used in realistic induction heating applications in cases where there are spurious oscillations due to irregular meshes.

\section{Conclusions and Discussion}\label{Sec:Concl}

We have presented the analysis and discretisation of the stationary two-dimensional solenoidal induction heating problem with Dirichlet boundary conditions on convex and non-convex polygonal domains. By proving that $H$ and $u$ were bounded in appropriate Sobolev norms, and in particular that $|\nabla H|^2$ belongs to $L^2(\Omega)$, we were able to provide a valid weak formulation for the problem at hand.  Based on this formulation, we propose first a standard Galerkin finite element method and proved that, for convex domains and under the assumption that meshes are quasiuniform,
solutions of the discrete  coupled problem, $(H_h, u_h)$, converged to $(H,u)$,
a solution of \eqref{eq:strong_solenoidal_theory}, as $h \rightarrow 0$. We then proposed a new bound-preserving finite element method and proved that,  for Lipschitz polygons and without any conditions on the mesh beyond shape regularity,  solutions to the discrete coupled system $(H_h^{\epsilon,k}, u_h^{(\epsilon,k)+})$ converge to $(H,u)$ as $h\rightarrow 0$ and $k \rightarrow \infty$, where $u_h^{(\epsilon,k)+}$ is the constrained part of the BPM solution of the heat equation. Moreover, we proved that for a fixed $k$, $u_h^{(\epsilon,k)+}$ converges to the projection of $u$ onto the closed convex subset $V^{\epsilon,k}$. These results were verified using numerical experiments on convergence when truncating at a non-physical bound, convergence on non-Delaunay meshes with boundary layers, and convergence of the coupled system,  for both the steady-state and the unsteady problem (even though the latter case lies outside the theoretical framework of this work).

These results present an extension to the current literature.  In particular, they extend results in \cite{Clain-Touzani-1997a} by allowing non-convex and polygonal domains. Convergence of FEMs for this problem is similarly limited to convex domains in the paper \cite{Parietti-Rappaz-1999} for the time-dependent problem. Therefore, by applying the BPM, and removing the requirements on the mesh and on the domain, we have presented a relaxation on the requirements for convergence for the coupled problem.

\section*{Acknowledgements}

The work of KM was funded by a EPSRC studentship at the University of Strathclyde, and Bifrangi UK. Ltd.  through the CORE funding for AFRC.  The work of AJS was funded by NSF grant DMS-2409918.
Part of this work was carried out while the authors were participating in the the Research-in-Groups programme: ``\emph{Approximation of renormalised solutions}'' at the International Centre for Mathematical Sciences, Edinburgh. ICMS support is gratefully acknowledged.  The authors also want to thank Dr. Aurik Andreu from the Advanced Forming Research Centre (AFRC), University of Strathclyde, for his help with the material properties used in the numerical experiment reported in Section 6.4.
KM and GRB wish to thank Tristan Pryer and Andreas Veeser for numerous and fruitful discussions about this work.

\appendix

\section{Proof of Lemma~\ref{conj:conjecture}}
\label{sec:RegularityObstacle}

The proof is a small variation on the well-known proofs of regularity for the solution of an obstacle problem, see \cite{Stampacchia-2000,MR1009785,MR880369} and in particular \cite[Section 3.1.14]{MR653144} for details about our penalisation approach. The only difference is that, in all references we were able to locate, the concern is with higher differentiability as well and, for this reason, the domain and problem data are assumed smoother than we do here.

We begin by recalling that $\Phi_\pm \in V^{\epsilon,k}$ and $\Phi = \Phi_+ - \Phi_-$, where
\[
  \Phi_+ \coloneqq \max\{ 0, \Phi \}, \qquad \Phi_- \coloneqq (-\Phi)_+.
\]
Thus, we may set $v = \Phi_+$ in \eqref{limit-ineq-uhatk} to get
\[
  - \| \nabla \Phi_- \|_{0,\Omega}^2 \geq (F,\Phi_-)_\Omega \geq 0,
\]
and so, $\Phi_- = 0$ almost everywhere. Meaning $\Phi = \Phi_+ \geq 0$.

Next, we sketch the proof of higher integrability of $\Phi$. As we stated above, the details are essentially the same as in the classical case. Let $\delta>0$ and let $\Phi_\delta \in H^1_0(\Omega)$ be the unique solution to
\[
  -\Delta \Phi_\delta + \frac1\delta \left( \Phi_\delta - k \right)_+ - \frac1\delta\left( \Phi_\delta + \epsilon \right)_- = F,
\]
in the weak sense. We may test the equation with $\frac1\delta \left( \Phi_\delta - k \right)_+$ to conclude
\[
  \left\| \frac1\delta \left( \Phi_\delta - k \right)_+\right\|_{0,\Omega} \leq 2 \| F \|_{L^2(\Omega)};
\]
see \cite[Theorem 3.1.10]{MR653144} for details. Similarly,
\[
  \left\| \frac1\delta \left( \Phi_\delta + \epsilon \right)_- \right\|_{0,\Omega} \leq 2 \| F \|_{L^2(\Omega)}.
\]
Define
\[
  G = F - \frac1\delta \left( \Phi_\delta - k \right)_+ + \frac1\delta\left( \Phi_\delta + \epsilon \right)_-,
\]
to see that
\[
  -\Delta \Phi_\delta = G, 
\]
and
\[
  \| G \|_{0,\Omega} \leq 5 \| F \|_{0,\Omega}.
\]
Since, by the Sobolev embedding \eqref{Ineq-Embedding} we have that 
\[
  \| F \|_{-1,p,\Omega} \leq C_{emb} \| F \|_{0,\Omega}, \qquad \forall p>1,
\]
we may then invoke \eqref{eq:JKWm1p} to conclude that
\[
  \| \nabla \Phi_\delta \|_{0,p,\Omega} \leq  5C_{emb}C_{JK}(p) \| F \|_{0,\Omega}, \qquad \forall p \in [P',P].
\]
The point of this estimate is that it is uniform in $\delta$, so we may extract a weakly convergent subsequence; which we do not relabel.

The final step is to recall that, as $\delta \to 0$, $\Phi_\delta \to \Phi$ in $H^1_0(\Omega)$. By lower semicontinuity with respect to weak convergence we arrive at
\[
  \| \nabla \Phi \|_{0,p,\Omega} \leq 5C_{emb}C_{JK}(p) \| F \|_{0,\Omega}, \qquad \forall p \in [P',P],
\]
which is the claimed higher integrability.

\bibliographystyle{plain}
\bibliography{refs}

@article{FEniCSx,
	author = {Baratta, I. A. and Dean, J. P. and Dokken, J. S. and Habera, M. and Hale, J. S. and Richardson, C. N. and Rognes, M. E. and Scroggs, M. W. and N. Sime, N. and Wells, G.N.},
	title = {DOLFINx: The next generation FEniCS problem solving environment,},
	journal = {Preprint},
	DOI = {doi.org/10.5281/zenodo.10447666},
	year = {2023},
	type = {Journal Article}
}

@article{Barrenechea-2024,
	author = {Barrenechea, G. R. and Georgoulis, E. H. and Pryer, T. and Veeser, A.},
	title = {A nodally bound-preserving finite element method},
	journal = {IMA Journal of Numerical Analysis},
	volume = {44},
	number = {4},
	pages = {2198-2219},
	year = {2024},
	type = {Journal Article}
}

@book{Barrenechea-John-Knobloch-2025,
	author = {Barrenechea, G. R. and John, V. and Knobloch, P.},
	title = {Monotone Discretizations for Elliptic Second Order Partial Differential Equations},
	publisher = {Springer Cham},
	volume = {61},
	series = {Springer Series on Computational Mathematics},
	year = {2025},
	type = {Book}
}

@article{Bay-Labbe-2003,
	author = {Bay, F. and Labbe, V. and Favennec, Y. and Chenot, J. L.},
	title = {A numerical model for induction heating processes coupling electromagnetism and thermomechanics},
	journal = {International Journal for Numerical Methods in Engineering},
	volume = {58},
	number = {6},
	pages = {839-867},
	year = {2003},
	type = {Journal Article}
}

@book{Brenner-Scott-2008,
	author = {Brenner, S. C. and Scott, L. R.},
	title = {The mathematical theory of finite element methods},
	publisher = {Springer},
	volume = {15},
	edition = {3rd},
	series = {Texts in Applied Mathematics},
	note = {3rd ed.},
	year = {2008},
	type = {Book}
}

@book {Brezis-2011,
    AUTHOR = {Brezis, H.},
     TITLE = {Functional analysis, {S}obolev spaces and partial differential
              equations},
    SERIES = {Universitext},
 PUBLISHER = {Springer, New York},
      YEAR = {2011},
     PAGES = {xiv+599},
      ISBN = {978-0-387-70913-0},
   MRCLASS = {35-01 (46-01 46E35 46N20 47F05)},
  MRNUMBER = {2759829},
MRREVIEWER = {Vicen\c tiu\ D.\ R\u adulescu},
}

@article{Chaboudez-1994,
	author = {Chaboudez, C. and Clain, S. and Glardon, R. and Rappaz, J. and Swierkosz, M. and Touzani, R.},
	title = {Numerical modelling of induction heating of long workpieces},
	journal = {IEEE Transactions on Magnetics},
	volume = {30},
	number = {6},
	pages = {5028-5037},
	year = {1994},
	type = {Journal Article}
}

@article{Clain-Touzani-1997a,
	author = {Clain, S. and Touzani, R.},
	title = {A Two-dimensional Stationary Induction Heating Problem},
	journal = {Mathematical Methods in the Applied Sciences},
	volume = {20},
	number = {9},
	pages = {759-766},
	year = {1997},
	type = {Journal Article}
}

@article{Clain-Touzani-1997b,
	author = {Clain, S. and Touzani, R.},
	title = {Solution of a two-dimensional stationary induction heating problem without boundedness of the coefficients},
	journal = {Mathematical Modelling and Numerical Analysis},
	volume = {31},
	number = {7},
	pages = {845-870},
	year = {1997},
	type = {Journal Article}
}

@book{Ern-Guermond-II-2021,
	author = {Ern, A. and Guermond, J. L.},
	title = {Finite Elements II - Galerkin Approximation, Elliptic and Mixed PDEs},
	publisher = {Springer},
	volume = {73},
	OPTedition = {1st 2021.},
	series = {Texts in Applied Mathematics},
	year = {2021},
	type = {Book}
}

@book{Ern-Guermond-I-2021,
	author = {Ern, A. and Guermond, J. L.},
	title = {Finite elements I - Approximation and interpolation},
	publisher = {Springer},
	volume = {72},
	series = {Texts in Applied Mathematics},
	year = {2021},
	type = {Book}
}

@article{GMSH,
	author = {Geuzaine, C. and Remacle, J. F.},
	title = {Gmsh: A 3-D finite element mesh generator with built-in pre- and post-processing facilities},
	journal = {International Journal for Numerical Methods in Engineering},
	volume = {79},
	number = {11},
	pages = {1309-1331},
	year = {2009},
	type = {Journal Article}
}

@article{Jensen-Malqvist-2012,
	author = {Jensen, M. and M\r{a}lqvist, A.},
	title = {Finite element convergence for the Joule heating problem with mixed boundary conditions},
	journal = {Bit Numerical Mathematics},
	volume = {53},
	number = {2},
	pages = {475-496},
	year = {2012},
	type = {Journal Article}
}

@article{Jensen-Malqvist-2022,
	author = {Jensen, M. and M\r{a}lqvist, A. and Persson, A.},
	title = {Finite element convergence for the time-dependent Joule heating problem with mixed boundary conditions},
	journal = {IMA Journal of Numerical Analysis},
	volume = {42},
	number = {1},
	pages = {199},
	year = {2022},
	type = {Journal Article}
}

@article{Jerison-Kenig-1995,
	author = {Jerison, D. and Kenig, C. E.},
	title = {The Inhomogeneous {D}irichlet Problem in {L}ipschitz Domains},
	journal = {Journal of Functional Analysis},
	volume = {130},
	number = {1},
	pages = {161-219},
	year = {1995},
	type = {Journal Article}
}

@book{Stampacchia-2000,
	author = {Kinderlehrer, David and Stampacchia, Guido},
	title = {An introduction to variational inequalities and their applications},
	publisher = {Society for Industrial and Applied Mathematics},
	volume = {31},
	series = {Classics in applied mathematics},
	year = {2000},
	type = {Book}
}

@article{Zhu-2006,
	author = {Loula, A. F. D. and Zhu, J.},
	title = {Mixed finite element analysis of a thermally nonlinear coupled problem},
	journal = {Numerical Methods for Partial Differential Equations. An International Journal},
	volume = {22},
	number = {1},
	pages = {180-196},
	year = {2006},
	type = {Journal Article}
}

@article{Zhu-2011,
title = {Mixed discontinuous {G}alerkin analysis of thermally nonlinear coupled problem},
journal = {Computer Methods in Applied Mechanics and Engineering},
volume = {200},
number = {13},
pages = {1479-1489},
year = {2011},
issn = {0045-7825},
doi = {https://doi.org/10.1016/j.cma.2010.12.009},
url = {https://www.sciencedirect.com/science/article/pii/S0045782510003543},
author = {J. Zhu and X. Yu and A.F.D. Loula},
}

@misc{MacKenzie-2025,
	author = {MacKenzie, K.},
	title = {Induction Heating},
	publisher = {GitHub},
	howpublished  ={\url{https://github.com/Katie795/induction_heating}},
	year = {2025},
	type = {Computer Program}
}

@article{Masse-Morel-Breville-1985,
	author = {Mass\'{e}, P. and Morel, B. and Breville, T.},
	title = {A finite element prediction correction scheme for magneto-thermal coupled problem during curie transition},
	journal = {IEEE Transactions on Magnetics},
	volume = {21},
	number = {5},
	pages = {1871-1873},
	year = {1985},
	type = {Journal Article}
}

@Article{Holst2010,
author={Holst, M.J. and Larson, M.G. and M{\aa}lqvist, A. and S{\"o}derlund, R.},
title={Convergence analysis of finite element approximations of the Joule heating problem in three spatial dimensions},
journal={BIT Numerical Mathematics},
year={2010},
month={Dec},
day={01},
volume={50},
number={4},
pages={781-795},
issn={1572-9125},
doi={10.1007/s10543-010-0287-z},
url={https://doi.org/10.1007/s10543-010-0287-z}
}

@article{Parietti-Rappaz-1998,
	author = {Parietti, C. and Rappaz, J.},
	title = {A quasi-static two-dimensional induction heating problem. {P}art {I}: modelling and analysis.},
	journal = {Mathematical Models and Methods in Applied Sciences},
	volume = {8},
	number = {6},
	pages = {1003-1021
	},
	year = {1998},
	type = {Journal Article}
}

@article{Parietti-Rappaz-1999,
	author = {Parietti, C. and Rappaz, J.},
	title = {A quasi-static two-dimensional induction heating problem. {P}art {II}: numerical analysis.},
	journal = {Math Models and Methods in Applied Sciences},
	volume = {9},
	number = {9  },
	pages = {1333-1350},
	year = {1999},
	type = {Journal Article}
}

@article{Rannacher-Scott-1982,
	author = {Rannacher, R. and Scott, R.},
	title = {Some optimal error-estimates for piecewise linear finite-element approximations},
	journal = {Mathematics of Computation},
	volume = {38},
	number = {158},
	pages = {437-445},
	year = {1982},
	type = {Journal Article}
}

@book{Raviart-Thomas-1992,
	author = {Raviart, P. A. and Thomas, J. M.},
	title = {Introduction \`{a} l'analyse num\'{e}rique des \'{e}quations aux d\'{e}riv\'{e}es partielles},
	publisher = {Masson},
	series = {Collection math\'{e}matiques appliqu\'{e}es pur la ma\^{i}trise sous la direction de P. G. Ciarlet et J. L. Lions},
	year = {1992},
	type = {Book}
}

@book{Roubicek-2013,
	author = {Roub\'{i}$\breve{\text{c}}$ek, T.},
	title = {Nonlinear partial differential equations with applications},
	publisher = {Birkh\"{a}user},
	volume = {153},
	edition = {2nd.},
	series = {International Series of Numerical Mathematics},
	year = {2013},
	type = {Book}
}

@book{Touzani-Rappaz-2014,
	author = {Touzani, R. and Rappaz, J.},
	title = {Mathematical Models for Eddy Currents and Magnetostatics: With Selected Applications},
	publisher = {Springer},
	series = {Scientific Computation},
	year = {2014},
	type = {Book}
}

@article {MR4748207,
    AUTHOR = {Diening, L. and Rolfes, J. and Salgado, A.J.},
     TITLE = {Pointwise gradient estimate of the {R}itz projection},
   JOURNAL = {SIAM J. Numer. Anal.},
  FJOURNAL = {SIAM Journal on Numerical Analysis},
    VOLUME = {62},
      YEAR = {2024},
    NUMBER = {3},
     PAGES = {1212--1225},
      ISSN = {0036-1429,1095-7170},
   MRCLASS = {65N30 (65N12 65N80)},
  MRNUMBER = {4748207},
       DOI = {10.1137/23M1571800},
       URL = {https://doi.org/10.1137/23M1571800},
}

@article {MR448949,
    AUTHOR = {Brezzi, F. and Hager, W.W. and Raviart, P.-A.},
     TITLE = {Error estimates for the finite element solution of variational
              inequalities},
   JOURNAL = {Numer. Math.},
  FJOURNAL = {Numerische Mathematik},
    VOLUME = {28},
      YEAR = {1977},
    NUMBER = {4},
     PAGES = {431--443},
      ISSN = {0029-599X,0945-3245},
   MRCLASS = {65N30},
  MRNUMBER = {448949},
MRREVIEWER = {Gilbert\ Strang},
       DOI = {10.1007/BF01404345},
       URL = {https://doi.org/10.1007/BF01404345},
}

@article {MR391502,
    AUTHOR = {Falk, R.S.},
     TITLE = {Error estimates for the approximation of a class of
              variational inequalities},
   JOURNAL = {Math. Comput.},
  FJOURNAL = {Mathematics of Computation},
    VOLUME = {28},
      YEAR = {1974},
     PAGES = {963--971},
   MRCLASS = {65K05 (35J30)},
  MRNUMBER = {391502},
MRREVIEWER = {F.\ Odeh},
       URL =
              {http://links.jstor.org/sici?sici=0025-5718(197410)28:128<963:EEFTAO>2.0.CO;2-2&origin=MSN},
}

@article {MR3393323,
    AUTHOR = {Nochetto, R.H. and Ot\'arola, E. and Salgado, A.J.},
     TITLE = {Convergence rates for the classical, thin and fractional
              elliptic obstacle problems},
   JOURNAL = {Philos. Trans. Roy. Soc. A},
  FJOURNAL = {Philosophical Transactions of the Royal Society A.
              Mathematical, Physical and Engineering Sciences},
    VOLUME = {373},
      YEAR = {2015},
    NUMBER = {2050},
     PAGES = {20140449, 14},
      ISSN = {1364-503X,1471-2962},
   MRCLASS = {65K15 (35J86)},
  MRNUMBER = {3393323},
MRREVIEWER = {Boualem\ Alleche},
       DOI = {10.1098/rsta.2014.0449},
       URL = {https://doi.org/10.1098/rsta.2014.0449},
}

@book {MR1009785,
    AUTHOR = {Friedman, A.},
     TITLE = {Variational principles and free-boundary problems},
   EDITION = {Second},
 PUBLISHER = {Robert E. Krieger Publishing Co., Inc., Malabar, FL},
      YEAR = {1988},
     PAGES = {x+710},
      ISBN = {0-89464-263-4},
   MRCLASS = {35R35 (35J85 49A29)},
  MRNUMBER = {1009785},
}

@book {MR880369,
    AUTHOR = {Rodrigues, J.-F.},
     TITLE = {Obstacle problems in mathematical physics},
    SERIES = {North-Holland Mathematics Studies},
    VOLUME = {134},
      NOTE = {Notas de Matem\'atica, 114. [Mathematical Notes]},
 PUBLISHER = {North-Holland Publishing Co., Amsterdam},
      YEAR = {1987},
     PAGES = {xvi+352},
      ISBN = {0-444-70187-7},
   MRCLASS = {35-02 (35J85 35R35 49A29)},
  MRNUMBER = {880369},
MRREVIEWER = {Maurizio\ Chicco},
}

@book {MR653144,
    AUTHOR = {Bensoussan, A. and Lions, J.-L.},
     TITLE = {Applications of variational inequalities in stochastic
              control},
    SERIES = {Studies in Mathematics and its Applications},
    VOLUME = {12},
      NOTE = {Translated from the French},
 PUBLISHER = {North-Holland Publishing Co., Amsterdam-New York},
      YEAR = {1982},
     PAGES = {xi+564},
      ISBN = {0-444-86358-3},
   MRCLASS = {49A29 (49-02)},
  MRNUMBER = {653144},
}

@article {MR2869035,
    AUTHOR = {Demlow, A. and Leykekhman, D. and Schatz, A. H. and Wahlbin,
              L. B.},
     TITLE = {Best approximation property in the {$W^{1}_{\infty}$} norm for
              finite element methods on graded meshes},
   JOURNAL = {Math. Comp.},
  FJOURNAL = {Mathematics of Computation},
    VOLUME = {81},
      YEAR = {2012},
    NUMBER = {278},
     PAGES = {743--764},
      ISSN = {0025-5718,1088-6842},
   MRCLASS = {65N30 (65N15 65N50)},
  MRNUMBER = {2869035},
MRREVIEWER = {Nicolae\ Pop},
       DOI = {10.1090/S0025-5718-2011-02546-9},
       URL = {https://doi.org/10.1090/S0025-5718-2011-02546-9},
}
\end{document}